\documentclass[11pt,reqno,a4paper]{amsart}
\usepackage[latin1]{inputenc}
\usepackage{tikz}
\usetikzlibrary{shapes,arrows}
\usetikzlibrary{arrows.meta}
\usepackage{forest}
\usepackage{tikz-qtree}

\usetikzlibrary{arrows,shapes,positioning,shadows,trees}

\usepackage[hidelinks]{hyperref}

\usepackage{etoolbox}
\usepackage{amsmath}
\usepackage{enumerate}
\usepackage{amssymb}
\usepackage{amscd}
\usepackage{amsthm}
\usepackage{amsfonts}
\usepackage{graphicx}
\usepackage[all,cmtip]{xy}
\usepackage{enumitem}

\patchcmd{\subsection}{-.5em}{.5em}{}{}
\patchcmd{\subsubsection}{-.5em}{.5em}{}{}

\usepackage{enumitem}

\makeatother
\usepackage{hyperref}

\numberwithin{equation}{section}

\newcommand{\SL}{\operatorname{SL}}

\newcommand{\SU}{\operatorname{SU}}

\newcommand{\GL}{\operatorname{GL}}

\newcommand{\cA}{\mathcal{A}}
\newcommand{\cB}{\mathcal{B}}

\newcommand{\cH}{\mathcal{H}}

\newcommand{\cL}{\mathcal{L}}

\newcommand{\cO}{\mathcal{O}}
\newcommand{\cP}{\mathcal{P}}

\newcommand{\cT}{\mathcal{T}}

\newcommand{\cX}{\mathcal{X}}

\newcommand{\bC}{\mathbb{C}}
\newcommand{\bD}{\mathbb{D}}

\newcommand{\bF}{\mathbb{F}}

\newcommand{\bN}{\mathbb{N}}

\newcommand{\bP}{\mathbb{P}}
\newcommand{\bQ}{\mathbb{Q}}
\newcommand{\bR}{\mathbb{R}}
\newcommand{\bS}{\mathbb{S}}

\newcommand{\bZ}{\mathbb{Z}}
\newcommand{\ra}{\rightarrow}

\newcommand{\qand}{\quad \textrm{and} \quad}

\newcommand\subsetsim{\mathrel{%
\ooalign{\raise0.2ex\hbox{$\subset$}\cr\hidewidth\raise-0.8ex\hbox{\scalebox{0.9}{$\sim$}}\hidewidth\cr}}}
\newcommand{\eps}{\varepsilon}

\DeclareMathOperator{\pr}{pr}

\DeclareMathOperator{\supp}{supp}

\DeclareMathOperator{\Stab}{Stab}

\DeclareMathOperator{\End}{End}
\DeclareMathOperator{\BndEnt}{BndEnt}

\DeclareMathOperator{\Ent}{Ent}
\DeclareMathOperator{\Prob}{Prob}

\DeclareMathOperator{\BS}{BS}

\DeclareMathOperator{\SubSum}{SubSum}

\theoremstyle{theorem}
\newtheorem{theorem}{Theorem}[section]
\newtheorem{corollary}[theorem]{Corollary}
\newtheorem{proposition}[theorem]{Proposition}
\newtheorem{lemma}[theorem]{Lemma}

\newtheorem*{conjecture}{Conjecture}

\theoremstyle{definition}
\newtheorem{definition}[theorem]{Definition}
\newtheorem{remark}[theorem]{Remark}

\newtheorem{example}{Example}[section]
\newtheorem{question}{Question}

\tikzstyle{decision} = [diamond, draw, fill=blue!20, 
    text width=4.5em, text badly centered, node distance=3cm, inner sep=0pt]
\tikzstyle{block} = [rectangle, draw, fill=blue!20, 
    text width=5em, text centered, rounded corners, minimum height=2em]
\tikzstyle{line} = [draw, -latex']
\tikzstyle{cloud} = [draw, ellipse,fill=red!20, node distance=3cm,
    minimum height=2em]

\renewcommand\labelenumi{(\roman{enumi})}
\renewcommand\theenumi\labelenumi

\usepackage{changepage}
\usepackage{mathtools}
\usepackage{graphicx}
\usepackage{calc}

\DeclarePairedDelimiterX\Set[2]{\{}{\}}{#1\,\delimsize\vert\,#2}

\newenvironment{myindentpar}[1]%
  {\begin{list}{}%
          {\setlength{\leftmargin}{#1}}%
          \item[]%
  }
  {\end{list}}

\begin{document}
\bibliographystyle{plain} % Choose Phys. Rev. style for bibliography

\title{Random walks on dense subgroups of locally compact groups}

\author{Michael Bj\"orklund}

\address{Department of Mathematics, Chalmers, Gothenburg, Sweden}
\email{micbjo@chalmers.se}

\author{Yair Hartman}
\address{Department of Mathematics, Ben Gurion University of the Negev, Be'er-Sheva, Israel}
\email{hartmany@bgu.ac.il}

\author{Hanna Oppelmayer}
\address{Department of Mathematics, Chalmers, Gothenburg, Sweden}
\email{hannaop@chalmers.se}
\date{}

\thanks{MB was partially supported by VR-grant 11253320, YH was partially supported by ISF grant 1175/18.}
\date{}

\keywords{Poisson boundaries of random walks on groups, Hecke pairs}

\subjclass[2010]{Primary: 37A40; Secondary: 05C81, 58J51}

\begin{abstract}
Let $\Gamma$ be a countable discrete group, $H$ a lcsc totally disconnected group and $\rho : \Gamma \ra H$ a homomorphism with dense image. 
We develop a general and explicit technique which provides, for every compact open subgroup $L < H$ and bi-$L$-invariant probability measure $\theta$
on $H$, a Furstenberg discretization $\tau$ of $\theta$ such that the Poisson boundary of $(H,\theta)$ is a $\tau$-boundary. Among other things, this technique allows us to construct examples of finitely supported random walks on certain lamplighter groups and solvable Baumslag-Solitar groups, whose Poisson boundaries are prime, but not $L^p$-irreducible for any $p \geq 1$, answering a conjecture of Bader-Muchnik in the negative. Furthermore, we 
give an example of a countable discrete group $\Gamma$ and two spread-out probability measures $\tau_1$ and $\tau_2$ on $\Gamma$ such that the boundary entropy spectrum of $(\Gamma,\tau_1)$ is an interval, while the boundary entropy spectrum of $(\Gamma,\tau_2)$ is a Cantor set. 
\end{abstract}

\maketitle

%\setcounter{tocdepth}{1}
%\tableofcontents

\section{Introduction}

\subsection{Furstenberg discretizations}
\label{subsec:furstenbergdisc}
Let $\mu$ be a Borel probability measure on a locally compact and second countable (lcsc) group $G$. We say that $\mu$ is \emph{spread-out} if it is absolutely continuous with respect to the Haar measure on $G$ and if its support generates $G$ as a semigroup. If $\mu$ is spread-out, we say that 
$\mu$ is a \emph{random walk} on $G$, and we refer to the pair $(G,\mu)$ as a \emph{measured group}. Let $(X,\cB_X)$ be a measurable space, 
endowed with a jointly Borel measurable action of $G$, and let $\xi$ be a probability measure on $\cB_X$. We say that $\xi$ is \emph{$\mu$-stationary}
if $\mu * \xi = \xi$. If $\xi$ is $\mu$-stationary, $(X,\xi)$ is a \emph{Borel $(G,\mu)$-space}.

Let $\Gamma$ be a countable discrete group and let $H$ be a lcsc group. Suppose that $\rho : \Gamma \ra H$ is a homomorphism with dense image. Suppose that $\theta$ is a spread-out Borel probability measure on $H$. We say that a spread-out probability measure $\tau$ on $\Gamma$ is a 
\emph{Furstenberg discretization} of $\theta$ (with respect to $\rho)$ if every Borel $(H,\theta)$-space is also a $(\Gamma,\tau)$-space, where $\Gamma$ acts via $\rho$ (equivalently, $\tau$ is a Furstenberg discretization of $\theta$ if the Poisson boundary of $(H,\theta)$ is $\tau$-stationary). The following theorem is due to Furstenberg \cite{Fur4}, see also \cite[Chapter VI, Proposition 4.1]{Ma} for a more detailed exposition.

\begin{theorem}[Furstenberg]
\label{thm0}
Suppose that $H$ is compactly generated and $\theta$ is compactly supported. Then $\theta$ admits a Furstenberg discretization $\tau$
with respect to $\rho$.
\end{theorem}

\begin{remark}
In \cite{Fur4}, Furstenberg only considered the setting when $\rho(\Gamma)$ is a a \emph{lattice subgroup} in $H$, so in particular, $\rho(\Gamma)$ is not 
dense in $H$. Hence our version of Theorem \ref{thm0} is not explicitly stated in \cite{Fur4} or \cite{Ma},  but it can be readily proved along the same lines as \cite[Chapter VI, Proposition 4.1]{Ma}, using that the Haar measurable function $L : G \ra [0,\infty)$ constructed there is left-invariant under $\rho(\Gamma)$, whence almost surely constant since $\rho(\Gamma)$ is dense in $G$. 

In the case when $\rho(\Gamma)$ is a lattice, Furstenberg instead employs an ingenious convexity argument to ensure that $L$ is almost surely constant. Once it has been established that $L$ is constant, the rest of the argument goes through as before.
\end{remark}

Before we move on to the main theme of this paper, let us briefly list some problematic aspects concerning the construction of Furstenberg discretizations in Theorem \ref{thm0}. 
\vspace{0.1cm}
\begin{itemize}
\item[\textsc{(A1)}] In general we cannot guarantee that $\tau$ is finitely supported (if $\Gamma$ is a finitely generated group). \vspace{0.2cm}
\item[\textsc{(A2)}] It is not known whether $\tau$ can always be chosen so that the Poisson boundary of $(H,\theta)$ is $\tau$-proximal (cf. Remark \ref{remark_prox}). \vspace{0.2cm}
\item[\textsc{(A3)}] Since the construction of $\tau$ from $\theta$ is very indirect, explicit computations with $\tau$ are usually very demanding. For 
instance, it seems like a daunting task in general to compute the Furstenberg entropy of a Borel $(H,\theta)$-space with respect to $\tau$. \vspace{0.1cm}
\end{itemize}

\begin{remark}
In the setting when $\rho(\Gamma)$ is a lattice, the aspect \textsc{(A2)} has been addressed (and affirmatively answered) in many special cases, see e.g.
\cite{BL, CM1, CM2, DD, LS}. In fact, in these cases, the Poisson boundary of $(H,\theta)$ is also the Poisson boundary of $(\Gamma,\tau)$. We stress that this may not longer be the case if $\rho(\Gamma)$ is dense in $H$. For instance, if $\Gamma = \SL_2(\bQ)$, $H = \SL_2(\bR)$ and $\rho$ is the standard inclusion, then for every spread-out probability measure $\theta$ on $H$ and for \emph{every} Furstenberg discretization $\tau$ of $\theta$, the Poisson boundary of $(H,\theta)$ is \emph{never} a maximal $\tau$-boundary (this is because the embeddings of $\Gamma$ into $\SL_2(\bQ_p)$ for different primes $p$ always contribute to the Poisson boundary of $\Gamma$; see e.g. \cite{BrSc} for a more detailed discussion about this point). 
\end{remark}

In this paper we study Furstenberg discretizations in the case when $H$ is \emph{totally disconnected}. Given a dense embedding $\rho$ of a countable 
group $\Gamma$ into a totally disconnected lcsc group $H$, and a compact and open subgroup $L$ of $H$, we shall introduce a convex set of probability 
measures on $\Gamma$ (called aborbing measures) and define an explicit \emph{surjective} affine map from the set of absorbing measures \emph{onto} the set of bi-$L$-invariant probability measures on $H$ such that every absorbing measure $\tau$ on $\Gamma$ is a Furstenberg discretization (with respect to $\rho$) of the image measure $\theta_\tau$ under this affine map. \\

Our novel construction is summarized in Theorem \ref{thm2} and Corollary \ref{cor1_Thm2} below, and some sample applications are given in Theorem \ref{thm3} and Theorem \ref{thm4}. 

\subsection{Hecke pairs and absorbing measures}

We shall now introduce the key players in this paper: \emph{absorbing measures}. Let $\Gamma$ be a countable group, and 
let $\Lambda$ be a subgroup of $\Gamma$. We denote by $\Prob(\Gamma)$ and $\Prob(\Gamma/\Lambda)$ the space of probability measures on $\Gamma$ and $\Gamma/\Lambda$ respectively,
and we write $\Prob(\Gamma/\Lambda)^\Lambda$ for the subset of $\Prob(\Gamma/\Lambda)$ consisting of $\Lambda$-invariant probability measures. 
Let $\alpha : \Gamma \ra \Gamma/\Lambda$ be the canonical projection map, and write $\alpha_*$ for the induced map between $\Prob(\Gamma)$ 
and $\Prob(\Gamma/\Lambda)$, given by
\[
\overline{\tau}(\gamma \Lambda) := \alpha_*\tau(\gamma \Lambda) = \sum_{\lambda \in \Lambda} \tau(\gamma \lambda), \quad \textrm{for $\gamma \Lambda \in \Gamma/\Lambda$}.
\]
The set $\Prob(\Gamma;\Lambda)$ of \emph{$\Lambda$-absorbing probability measures on $\Gamma$} is defined by
\[
\Prob(\Gamma ;\Lambda) = \big\{ \tau \in \Prob(\Gamma) \, : \, \alpha_*\tau \in \Prob(\Gamma/\Lambda)^\Lambda \big\}.
\] 
Note that if $\Lambda$ is \emph{normal}, then every probability measure on $\Gamma$ is $\Lambda$-absorbing. \\

We say that $(\Gamma,\Lambda)$ is a \emph{Hecke pair} if for every $\gamma \in \Gamma$, the subgroup $\Lambda \cap \gamma \Lambda \gamma^{-1}$ has finite index in $\Lambda$; or equivalently, if every $\Lambda$-orbit in $\Gamma/\Lambda$ is finite. In this case, it is also common to refer to $\Lambda$ as an \emph{almost normal} (or \emph{commensurated}) subgroup of $\Gamma$. \\

We collect some basic properties of $\Prob(\Gamma ; \Lambda)$ in the following theorem. 

\begin{theorem}
\label{thm1}
Let $\Gamma$ be a countable group and let $\Lambda$ be a subgroup of $\Gamma$.
\vspace{0.1cm}
\begin{enumerate}
\item[\textsc{(i)}] For all $\tau_1, \tau_2 \in \Prob(\Gamma ; \Lambda)$,
\[
\alpha_*(\tau_1 * \tau_2)(\gamma \Lambda) = \sum_{\eta \Lambda} \overline{\tau}_1(\eta \Lambda) \, \overline{\tau}_2(\eta^{-1} \gamma \Lambda), 
\quad \textrm{for every $\gamma \Lambda \in \Gamma/\Lambda$}.
\]
In particular, $(\Prob(\Gamma ; \Lambda),*)$ is a monoid, where $\delta_e$ is the identity element. \vspace{0.2cm}

\item[\textsc{(ii)}] If $\Gamma$ acts by linear isometries on a Banach space $E$ and the dual action  preserves a weak*-closed convex
subset $C$ of $E^*$, then 
\[
\Prob(\Gamma ; \Lambda) * C^\Lambda \subseteq C^\Lambda,
\]
where $C^\Lambda$ denotes the (possibly empty) set of $\Lambda$-invariant points in $C$. In particular, if $C$ is weak*-compact
and $\Lambda$ is amenable, then for every $\tau \in \Prob(\Gamma ; \Lambda)$, there exists $c \in C^\Lambda$ such that  $\tau * c = c$. \vspace{0.2cm}

\item[\textsc{(iii)}] If $(\Gamma,\Lambda)$ is a Hecke pair, then for every subset $S \subset \Gamma$, there exist 
\vspace{0.1cm}
\begin{enumerate}
\item[(i)] a subset $\widetilde{S} \subset \Gamma$, \vspace{0.1cm}

\item[(ii)] an affine map $\Prob(S) \ra \Prob(\Gamma ; \Lambda) \cap \Prob(\widetilde{S}), \enskip \tau \mapsto \widetilde{\tau}$,
\end{enumerate} 
\vspace{0.1cm}
such that $\supp(\tau) \subseteq \supp(\widetilde{\tau})$. If $S$ is finite, then $\widetilde{S}$ can be chosen finite. In particular, if $\tau$
is spread-out, then so is $\widetilde{\tau}$.
\end{enumerate}
\end{theorem}

The main point of \textsc{(III)} is to show that if $\Gamma$ is finitely generated and $(\Gamma,\Lambda)$ is a Hecke pair, then finitely supported, spread-out and $\Lambda$-absorbing measures exist in abundance. 

%In the next sub-section, where we formulate our first main result, we shall also discuss a general construction of Hecke pairs due to Tzanev \cite{Tza}, based on previous works of Schlichting \cite{Sch}, and present several examples of Hecke pairs.

\subsection{Hecke measured groups and their completions}

Hecke pairs can be constructed along the following lines. Given \vspace{0.1cm}
\begin{itemize}
\item a countable discrete group $\Gamma$, \vspace{0.1cm}
\item a lcsc \emph{totally disconnected group} $H$ and a compact open subgroup $L < H$,  \vspace{0.1cm}
\item a homomorphism $\rho : \Gamma \ra H$ with dense image, \vspace{0.1cm}
\end{itemize}
we set $\Lambda := \rho^{-1}(L) < \Gamma$, and note that 
\[
\Lambda \cap \gamma \Lambda \gamma^{-1} = \rho^{-1}(L \cap \rho(\gamma)L\rho(\gamma)^{-1}), \quad \textrm{for all $\gamma \in \Gamma$}.
\]
Since $L$ is compact and open, the intersection $L \cap \rho(\gamma)L\rho(\gamma)^{-1}$ is a compact and open subgroup of $L$, and has thus 
finite index in $L$ for every $\gamma \in \Gamma$, whence $(\Gamma,\Lambda)$ is a Hecke pair. We say that $(H,L,\rho)$ is a \emph{completion triple} of the Hecke pair $(\Gamma,\Lambda)$. Conversely, Tzanev \cite{Tza} has proved (based on some ideas of Schlicting \cite{Sch}) that if $(\Gamma,\Lambda)$ is a Hecke pair, then there is always a completion triple $(H,L,\rho)$ of $(\Gamma,\Lambda)$. This completion triple is often referred to as the \emph{Schlicting (or relatively profinite) completion} of $(\Gamma,\Lambda)$ in the literature, and we shall discuss several examples below. \\

The key point of our first main result, and its corollaries, is that if $(\Gamma,\Lambda)$ is a Hecke pair and $(H,L,\rho)$ is a completion triple of $(\Gamma,\Lambda)$, then there is a natural affine map between spread-out and $\Lambda$-absorbing probability measures on $\Gamma$ and spread-out and bi-$L$-invariant probability measures on $H$, which
\begin{itemize}
\item respects convolutions, \vspace{0.1cm}
\item for every $\Lambda$-absorbing spread-out probability measure $\tau$ on $\Gamma$ produces a spread-out bi-$L$-invariant probability measure $\theta_\tau$  on $H$ such that $\tau$ is a Furstenberg discretization of $\theta_\tau$. 
\vspace{0.1cm}
\item induces a $\Gamma$-equivariant and measure-preserving map between the Poisson boundary of $(\Gamma,\tau)$ and the Poisson boundary
of $(H,\theta_\tau)$. In some cases, this map is a measurable $\Gamma$-isomorphism.
\end{itemize}

\begin{remark}
In his Ph.D.-thesis \cite[Subsection 7.2.1]{Cr}, Creutz constructs (in the setting described above) finitely supported Furstenberg discretizations on $\Gamma$ for \emph{certain} compactly supported bi-$L$-invariant probability measures on $H$. His construction is however quite different from ours, and does not seem to induce a natural map between the respective Poisson boundaries (this is also not needed for the applications that he had in mind).
\end{remark}

We denote by $\Prob(H,L)$ the space of bi-$L$-invariant probability measures on $H$, and note that this set is clearly closed under convolution and 
that the Haar probability measure $m_L$ on $L$ is the neutral element with respect to convolution. In particular, $(\Prob(H,L),*)$ is a monoid. Since $L$ is open in $H$, every bi-$L$-invariant measure on $H$ is automatically absolutely continuous with respect to the Haar measure class on $H$.

\begin{theorem}
\label{thm2}
Let $(\Gamma,\Lambda)$ be a Hecke pair and $(H,L,\rho)$ a completion triple of $(\Gamma,\Lambda)$. There exists an affine surjective monoid homomorphism
\[
(\Prob(\Gamma ; \Lambda),*) \longrightarrow (\Prob(H,L),*), \enskip \tau \mapsto \theta_\tau
\]
with the following properties: 
\vspace{0.1cm}
\begin{enumerate}
\item[\textsc{(P1)}] for every $\tau \in \Prob(\Gamma ; \Lambda)$ and right $L$-invariant $\varphi \in C(H) \cap \cL^1(H,\theta_\tau)$,
\[
\sum_{\gamma \in \Gamma} \varphi(\rho(\gamma)) \, \tau(\gamma) = \int_H \varphi(h) \, d\theta_\tau(h).
\]
In particular, for every $\theta_\tau$-integrable homomorphism $\varphi : H \ra \bR$, we have 
\[
\tau(\varphi \circ \rho) = \theta_\tau(\varphi).
\]
\item[\textsc{(P2)}] for every $\tau \in \Prob(\Gamma ; \Lambda)$,
\[
\supp \theta_\tau = L\rho(\supp \tau)L.
\]
In particular, if $\tau$ is spread-out, then so is $\theta_\tau$, and if $\tau$ has finite support, then $\theta_\tau$ has compact support.
\vspace{0.2cm}

\item[\textsc{(P3)}] for every Borel $H$-space $X$ and $L$-invariant Borel measure $\xi$ on $X$, 
\[
\tau * \xi = \theta_\tau * \xi, \quad \textrm{for all $\tau \in \Prob(\Gamma ; \Lambda)$},
\]
where $\Gamma$ acts on $X$ via $\rho$. In particular, every $(H,\theta_\tau)$-space is also a $(\Gamma,\tau)$-space. \vspace{0.2cm}

\item[\textsc{(P4)}] for every $\tau \in \Prob(\Gamma ; \Lambda)$ and Borel $(H,\theta_\tau)$-space $(X,\xi)$,
\[
\textrm{$\xi$ is $\theta_\tau$-proximal} \iff \textrm{$\xi$ is $\tau$-proximal}.
\]
\end{enumerate}
\end{theorem}

\begin{remark}
Property \textsc{(P2)} tells us that spread-out measures are mapped to spread-out measures, so in combination with Property \textsc{(P3)} we can conclude that $\theta_\tau$ is a Furstenberg discretization of $\tau$ whenever $\tau$ is spread-out. We stress that we do \emph{not} claim that every compactly supported measure in $\Prob(H,L)$ must necessarily be the image of a finitely supported probability measure on $\Gamma$. Property \textsc{(P4)} implies that the Poisson boundary of $(H,\theta_\tau)$, viewed as a $(\Gamma,\tau)$-space, is a $\tau$-boundary, and thus provides an answer to \textsc{(A2)} in our setting. 
\end{remark}

\begin{definition}[Hecke measured group and its Hecke completion]
Let $(\Gamma,\Lambda)$ be a Hecke pair, and $(H,L,\rho)$ a completion triple of $(\Gamma,\Lambda)$. If $\tau$ is a spread-out and $\Lambda$-absorbing probability measure on $\Gamma$, we say that $(\Gamma,\tau)$ is a
\emph{Hecke measured group}, and refer to $(H,\theta_\tau)$ as the \emph{Hecke completion of $(\Gamma,\tau)$ with respect to $\rho$}.
\end{definition}

Let us now discuss some functorial properties of the map $\tau \mapsto \theta_\tau$ in Theorem \ref{thm2}, restricted to spread-out measures. 
The following corollary will be proved in Section \ref{sec:proofCor1Thm2}, and we retain the notation from Theorem \ref{thm2}. 
We refer to Section \ref{sec:prelcpt} and Section \ref{sec:prelBorel} for definitions. 

\begin{corollary}
\label{cor1_Thm2}
Let $\tau$ be a spread-out measure in $\Prob(\Gamma ; \Lambda)$, and denote by
$(B_\tau,\nu_\tau)$ and $(B_{\theta_\tau},\nu_{\theta_\tau})$ the Poisson boundaries of $(\Gamma,\tau)$ and $(H,\theta_\tau)$
respectively, where $\Gamma$ acts on $B_{\theta_\tau}$ via $\rho$. \vspace{0.1cm}
\begin{enumerate}
\item[\textsc{(i)}] $(B_{\theta_\tau},\nu_{\theta_\tau})$ is a $\tau$-boundary. \vspace{0.3cm}

\item[\textsc{(ii)}] If $M$ is a compact $\Gamma$-space, and $\eta$ is a $\Lambda$-invariant and $\tau$-stationary probability measure on $M$,
then there is a measure-preserving $\Gamma$-map 
\[
(B_{\theta_\tau},\nu_{\theta_\tau}) \ra (\Prob(M),\eta^*),
\]
where $(\Prob(M),\eta^*)$ denotes the canonical quasi-factor of $(M,\eta)$.
In particular, every $\tau$-proximal and $\Lambda$-invariant Borel $(\Gamma,\tau)$-space is measurably $\Gamma$-isomorphic to a $\theta_\tau$-boundary, viewed as a $(\Gamma,\tau)$-space. \vspace{0.3cm}

\item[\textsc{(iii)}] If $(B_{\theta_\tau},\nu_{\theta_\tau})$ is measurably $\Gamma$-isomorphic to $(B_\tau,\nu_\tau)$, then $\Lambda$ is amenable. \vspace{0.3cm}

\item[\textsc{(iv)}] If $\Lambda$ is amenable and $(B_\tau,\nu_\tau)$ admits a uniquely $\tau$-stationary compact model, 
then $(B_\tau,\nu_\tau)$ is measurably $\Gamma$-isomorphic to $(B_{\theta_\tau},\nu_{\theta_\tau})$.\vspace{0.3cm}
\end{enumerate}
\end{corollary}

\subsection{Main lines of investigations}

In many situations, Theorem \ref{thm2} and Corollary \ref{cor1_Thm2} offer a possibility to study certain aspects of $(\Gamma,\tau)$-spaces 
through the lenses of $(H,\theta_\tau)$-spaces, which are often better behaved (the boundary theory of such measured groups has been systematically developed by Kaimanovich and Woess in \cite{KaWo}). 

The benefits of such a change of perspective are especially rewarding when the Poisson boundary of 
$(H,\theta_\tau)$ is essentially $H$-transitive, that is to say, modulo null sets, measurably $H$-isomorphic to a quotient space of the form $H/P$ for some \emph{closed} subgroup $P < H$, where $H/P$ is endowed with its unique $H$-invariant measure class. Our main applications (Theorem \ref{thm3} and
Theorem \ref{thm4}) indeed take place in this scenery (although this point of view is somewhat hidden in the arguments). 

Before
we describe these applications in detail, we will briefly provide some general background to the lines of investigations that we will pursue in this paper, and to the questions that we will answer.
\vspace{0.4cm}

\noindent \textbf{Existence of prime measured groups} \\

Let $(\Gamma,\tau)$ be a countable measured group and let $(X,\xi)$ be a non-trivial Borel $(\Gamma,\tau)$-space. We say that $(X,\xi)$ is \emph{prime} if every measure-preserving $\Gamma$-map from $(X,\xi)$ to another $(\Gamma,\tau)$-space is either a measurable $\Gamma$-isomorphism
or essentially trivial, and we say that a measured group $(\Gamma,\tau)$ is \emph{prime} if the associated Poisson boundary is  prime.

As far as we know, before this paper (more specifically, Theorem \ref{thm3} below), not a single example of a prime (countable) measured group 
was known. Furthermore, the only \emph{explicit} examples in the literature of $\tau$-proximal and prime Borel $(\Gamma,\tau)$-spaces that we could find 
(modulo small variations) were:
\vspace{0.1cm}
\begin{itemize}
\item Let $\Gamma$ be a lattice in a simple real Lie group $G$ of real rank at least two, and let $Q$ denote the \emph{maximal} parabolic 
subgroup of $G$. Let $\tau$ be a spread-out probability measure on $\Gamma$ such that the (unique) $\tau$-stationary Borel probability measure $\nu$
on $G/Q$  is absolutely continuous with respect to the unique $G$-invariant measure class (the existence of such a measure is guaranteed by Theorem \ref{thm0}). Margulis (see e.g. \cite[Theorem 8.1.4]{Zi}) has proved that $(G/Q,\nu)$ is a prime (non-maximal) $\tau$-boundary. This result is a key ingredient in the proof of Margulis' Normal Subgroup Theorem. \vspace{0.1cm}

\item If $\Gamma$ has Property (T), then Nevo \cite[Theorem 4.3]{N} has proved that for every (say, symmetric and finitely supported) spread-out probability measure $\tau$ on $\Gamma$, the
measured group $(\Gamma,\tau)$ always admits (at least one) prime $\tau$-boundary. \vspace{0.1cm}
\end{itemize}

Using Corollary \ref{cor1_Thm2}, we can provide (see Section \ref{sec:cor2thm2}) a new class of examples of $\tau$-proximal and prime $(\Gamma,\tau)$-spaces when $\Gamma$ is a free group of finite rank. 

\begin{corollary}
\label{cor2_thm2}
There exist an integer $r \geq 2$ and a finitely supported spread-out probability measure $\tau$ on a free group $\Gamma$ of rank $r$ such
that the measured group $(\Gamma,\tau)$ admits a prime and essentially free $\tau$-boundary.
\end{corollary}

\begin{remark}
%One of the key point of Corollary \ref{cor2_thm2} is that the $\tau$-boundary is essentially free. Indeed, if $\Gamma$ is a free group of rank $r$, and $\Xi$
%is an $r$-generated group with Property (T), then, by the universal property of free groups, there is a surjective homomorphism $p$ from $\Gamma$ onto
%$\Xi$, which induces a measure-preserving $\Gamma$-equivariant map between the Poisson boundary of $(\Gamma,\tau)$ and the Poisson boundary 
%of $(\Xi,p_*\tau)$. By Nevo's result above, $(\Xi,p_*\tau)$ has a prime $p_*\tau$-boundary, which is then clearly also a prime $\tau$-boundary. However, on this $\tau$-boundary, the (non-trivial) kernel of $p$ acts trivially, and thus the action is not essentially free.

%One of the key points in this corollary is that $\tau$ can be chosen finitely supported. This uses Property \textsc{(II)} in 
%Theorem \ref{thm1}, together with a dense embedding of $\Gamma$ into $SL_2(\bQ_p)$ for some prime number $p$. The prime and essentially free
%$\tau$-boundary in the corollary will be the Poisson boundary of $\SL_2(\bQ_p)$ with respect to a compactly supported bi-$SL_2(\bZ_p)$-invariant measure (the Hecke completion of $\tau$). The fact that this is a $\tau$-boundary follows from \textsc{(I)} in Corollary \ref{cor1_Thm2}.

In \cite{Bou}, Bourgain constructs a dense free subgroup $\Gamma$ of finite rank in $\SU(1,1)$ such that the Lebesgue measure $\nu$ on the boundary 
$\partial \bD$ of the unit disc $\bD$ in the complex plane is the unique $\tau$-stationary probability on $\partial \bD$, where $\tau$ denotes the uniform probability measure $\tau$ on a set of free generators of $\Gamma$. It follows from quite general principles that $(\partial \bD, \nu)$ is a $\tau$-boundary,
and its primeness and essential freeness can be established along the same lines as in the proof of Corollary \ref{cor2_thm2}. Similar constructions can also
be found in \cite{BPS} and \cite{BQ}.
\end{remark}

Furthermore, Corollary \ref{cor1_Thm2} also provides the following criterion for when a measured group is \emph{not} prime (see Section \ref{sec:cor3thm2}).

\begin{corollary}
\label{cor3_thm2}
Let $(\Gamma,\Lambda)$ be a  Hecke pair, and suppose that $\Lambda$ is neither amenable nor co-amenable in $\Gamma$. 
Then, for every spread-out and $\Lambda$-absorbing probability measure on $\Gamma$, the measured group $(\Gamma,\tau)$ is  \emph{not} prime.
\end{corollary}

%\begin{remark}
%In Subsection \ref{subsec:somequestions} below (before Question 5), we discuss primeness in a special setting where $\Lambda$ is non-amenable, 
%\emph{but} co-amenable. 
%%(in which case the group $H$ the Schlicting completion triple $(H,L,\rho)$ of $(\Gamma,\Lambda)$ is amenable).
%\end{remark}

\vspace{0.2cm}

\noindent\textbf{$L^p$-irreducibility of boundary representations}\\

Let $(\Gamma,\tau)$ be a measured group and let $(X,\xi)$ be a Borel $(\Gamma,\tau)$-space. For $p \in [1,\infty)$, we denote by $\cO(L^p(X,\xi))$ the
group of orthogonal transformations on $L^p(X,\xi)$, and we define the map $\sigma_p : \Gamma \ra \cO(L^p(X,\xi))$ by
\[
\sigma_p(\gamma)f = \Big( \frac{d\gamma \xi}{d\xi} \Big)^{1/p} f(\gamma^{-1} \cdot), \quad \textrm{for $\gamma \in \Gamma$ and $f \in L^p(X,\xi)$}.
\]
One readily checks that $\sigma_p$ is a homomorphism. We say that $(X,\xi)$ is \emph{$L^p$-irreducible} 
if for every non-zero element $f \in L^p(X,\xi)$, the linear span of the set $\{ \sigma_p(\gamma)f \, \mid \, \gamma \in \Gamma \big\}$ is norm-dense in $L^p(X,\xi)$. We can of course extend the definition of the $L^p$-representation $\sigma_p$ to $p = \infty$ by setting $\sigma_\infty(\gamma)f = f \circ \gamma^{-1}$ 
for $\gamma \in \Gamma$ and $f \in L^\infty(X,\xi)$. However, since $\Gamma$ is countable and $L^\infty(X,\xi)$ is non-separable in the norm topology 
(if the support of $\xi$ is infinite), (infinite) Borel $(\Gamma,\tau)$-spaces are never $L^\infty$-irreducible. \\

In \cite{BM}, Bader and Muchnik formulated the following influential conjecture:

\begin{conjecture}
The Poisson boundary of a measured group is $L^2$-irreducible.
\end{conjecture}

Bader and Muchnik \cite{BM} proved their conjecture for all Gromov hyperbolic groups, see also \cite{BD,CM2, D} for various extensions. We are not
aware of any previous investigations into $L^p$-irreducibility for $p \neq 2$. In 
Theorem \ref{thm3} below we provide explicit (solvable) measured groups which are \emph{not} $L^p$-irreducible for any $p \in [1,\infty)$,
thereby answering the conjecture above in the negative. The construction of our counterexample depends crucially on the work \cite{Kai0} of Kaimanovich, applied in combination with Theorem \ref{thm2} and Corollary \ref{cor1_Thm2} above. \\

\noindent \textbf{The topological structures of boundary entropy spectra}\\

Let $(\Gamma,\tau)$ be a countable measured group and let $(X,\xi)$ be a Borel $(\Gamma,\tau)$-space. The \emph{Furstenberg entropy} 
$h_\tau(X,\xi)$ is given by
\[
h_\tau(X,\xi) = \sum_{\gamma \in \Gamma} \tau(\gamma) \int_X -\log \frac{d\gamma^{-1}\xi}{d\xi}(x) \, d\xi(x),
\]
whenever the integral is well-defined.
The \emph{entropy spectrum} $\Ent(\Gamma,\tau)$ is defined by
\[
\Ent(\Gamma,\tau) = \{ h_\tau(X,\xi) \, \mid \, \textrm{$(X,\xi)$ is an ergodic Borel $(\Gamma,\tau)$-space} \big\}, 
\]
and the \emph{boundary entropy spectrum} $\BndEnt(\Gamma,\tau)$ is defined by
\[
\BndEnt(\Gamma,\tau) = \{ h_\tau(X,\xi) \, \mid \, \textrm{$(X,\xi)$ is a $\tau$-proximal Borel $(\Gamma,\tau)$-space} \big\}.
\]
Since every $\tau$-proximal Borel $(\Gamma,\tau)$-space is ergodic, we have $\BndEnt(\Gamma,\tau) \subseteq \Ent(\Gamma,\tau)$. \\

Starting with the discussions in \cite{NZ} by Nevo and Zimmer, the sets $\End(\Gamma,\tau)$ and $\BndEnt(\Gamma,\tau)$ have been subject to intense studies. It readily follows
from \cite{BoHaTa} that under mild assumptions, both $\Ent(\Gamma,\tau)$ and $\BndEnt(\Gamma,\tau)$ are continuous images 
of $G_\delta$-sets, whence analytic subsets of $[0,\infty)$. It has been conjectured (see for instance \cite{BuLuTa}) that under mild assumptions, 
$\BndEnt(\Gamma,\tau)$ is always a closed subset. 

Bowen proved in \cite{Bo} that for certain probability measures $\tau$ on a free group $\Gamma$ of finite rank, $\Ent(\Gamma,\tau)$ is a closed interval of the form $[0,h(\tau)]$, where $h(\tau)$ is the Furstenberg entropy of the
Poisson boundary of $(\Gamma,\tau)$. On the other hand, very little is known for $\BndEnt(\Gamma,\tau)$ in the same setting (although 
Tamuz and Zheng \cite{TaZh} have recently proved that it at least contains a Cantor set). 

There are at least two reasons for why the analysis of the boundary entropy spectrum of a measured group is difficult. Firstly, it is often hard to 
find a manageable parameterization of the set of all $\tau$-proximal Borel $(\Gamma,\tau)$-spaces. Secondly, even if such a parameterization
is available, computing the corresponding Furstenberg entropies is usually quite demanding. Using some results from \cite{BS}, we shall 
in Theorem \ref{thm4} below construct from a given prime measured group (for instance one of the ones provided by Theorem \ref{thm3}), another measured group, whose boundary entropy spectrum can be computed explicitly. As far as we know, this is the first \emph{explicit} realization of an infinite 
boundary 
entropy spectrum of a measured group. We also exhibit a curious phenomenon: a single countable group can admit different spread-out 
probability measures with radically different boundary entropy spectra; indeed, we construct in  Theorem \ref{thm4} below a countable group $\Gamma$ 
and two different
spread-out probability measures $\tau$ and $\tau'$ on $\Gamma$ so that $\BndEnt(\Gamma,\tau)$ is a Cantor set, while $\BndEnt(\Gamma,\tau')$
is a closed interval.

\subsection{Hecke measured groups from non-archimedean local fields}
\label{subsec:grpsaffine}

In what follows, let $(K,|\cdot|)$ be a non-archimedean local field with Haar measure $m_K$. We
set
\[
\cO = \{ x \in K \, \mid \, |x| \leq 1 \big\} \qand \cP = \{ x \in K \, \mid \, |x| < 1 \big\},
\]
and assume that the residue field $k = \cO/\cP$ is a finite cyclic group of (prime) order $q$. Pick a non-zero element $x_o \in \cO$ such that 
\[
S := \{0,x_o,\ldots,(q-1)x_o\} \subset \cO
\] 
is a set of representatives for $\cO/\cP$. We fix a uniformizer $\varpi$ of $K$. Then $|\varpi| = \frac{1}{q}$, and by \cite[Proposition 4.17(ii)]{RV}, 
every $x \in K$ can be uniquely expressed as a convergent power series in $\varpi$ of the form
\[
x = \sum_{j=n}^\infty a_j \varpi^j, \quad \textrm{for some $n \in \bZ$},
\]
where $a_j \in S$ for all $j$. Since $K$ is a field and $S$ has the special form above, we see that the additive group $\Xi$ which is generated by $0$ and 
all powers of $\varpi$ is dense in $K$. We denote by $\Xi_o$ the subgroup of $\Xi$ which is generated by $0$ and all non-negative powers of $\varpi$, and note that $\Xi_o = \Xi \cap \cO$ is dense in $\cO$. One readily checks that
\[
\varpi \, \Xi_o \subset \Xi_o \subset \varpi^{-1} \, \Xi_o \qand |\Xi_o/\varpi \,\Xi_o| < \infty,
\]
which in particular implies that
\[
\Gamma = \Xi \rtimes \langle \varpi \rangle \qand \Lambda = \Xi_o \rtimes \{1\},
\]
where $\langle \varpi \rangle$ denotes the cyclic (multiplicative) group generated by $\varpi$ (which acts on $\Xi$ by multiplication)
is a Hecke pair. Furthermore, 
\[
H = K \rtimes \langle \varpi \rangle \qand L = \cO \rtimes \{1\} \qand \rho = \textrm{id},
\]
is a completion triple of $(\Gamma,\Lambda)$. It is not difficult to show that $\Gamma$ is a finitely generated group and that $H$ is a lcsc 
compactly generated group. We shall think of $\Gamma$ as a dense subgroup of $H$. Moreover, $H$ acts jointly continuously and transitively on $K$ by 
\[
(x,\varpi^n)y = x + \varpi^n y, \quad \textrm{for $(x,\varpi^n) \in H$ and $y \in K$}.
\] 
In particular, $K \cong H/P$, where $P = \{0\} \rtimes \langle \varpi \rangle$.  We write $\pr_{\bZ}$ for the surjective homomorphism
\begin{equation}
\label{def_prZ}
\pr_{\bZ} : H \ra \bZ, \enskip (x,\varpi^n) \mapsto n.
\end{equation}
Our third main theorem now reads as follows. 

\begin{theorem}
\label{thm3}
Let $\tau$ be a finitely supported, spread-out and $\Lambda$-absorbing probability measure on $\Gamma$. 
Suppose that 
\begin{equation}
\label{negdrift_intro}
\sum_{n \in \bZ} n \, (\pr_{\bZ})_*\tau(n) < 0.
\end{equation}
Then the following holds.
\vspace{0.1cm}
\begin{enumerate}
\item[\textsc{(i)}] There exists a unique $\tau$-stationary probability measure $\nu$ on $K$. \vspace{0.2cm}
\item[\textsc{(ii)}] $\nu$ is $\cO$-invariant, $H$-quasi-invariant and absolutely continuous with respect to $m_K$ with an $\cO$-invariant (thus continuous) everywhere positive density. \vspace{0.2cm}
\item[\textsc{(iii)}] The Borel $(\Gamma,\tau)$-space $(K,\nu)$ is prime and $\tau$-proximal. \vspace{0.2cm}
\item[\textsc{(iv)}] For every $p \in [1,\infty)$, the quasi-regular $\Gamma$-representation on $L^p(K,\nu)$ is \emph{not} irreducible.
\end{enumerate}
\end{theorem}

\begin{remark}
The existence, uniqueness and $\tau$-proximality of $\nu$ (assuming that $\tau$ is finitely supported, spread-out and satisfies \ref{negdrift_intro}) 
is a standard result (see e.g. \cite{Kai0, Br0, Br}). Furthermore, \textsc{(iii)} and \textsc{(iv)} hold whenever $\nu$ is absolutely continuous with respect 
to the Haar measure $m_K$. Only \textsc{(ii)}, which is the key point of the theorem, exploits the assumption that $\tau$ is $\Lambda$-absorbing.
This raises the question how essential $\Lambda$-absorption is to ensure that $\nu$ is absolutely continuous with respect to $m_K$. In general this is quite a subtle question (for archimedean fields, this is closely related to the classical line of research, initated by Erd\"os and Wintner, pertaining to absolute continuity of Bernoulli convolutions). Recently, Brieussel and Tanaka \cite{BrTa} have developed a technique to construct finitely supported and spread-out probability measures on groups of
\emph{real} affine transformations, whose actions on the real line admit unique stationary measures which are \emph{singular} with respect to the Lebesgue
measure. We plan to extend their techniques to non-archimedean fields  in future works, thus showing that $\Lambda$-absorption of $\tau$ is crucial
to establish \textsc{(ii)}.
\end{remark}

\begin{remark}
Before we proceed, we show how one can explicitly construct finitely supported and spread-out $\tau$ in $\Prob(\Gamma ; \Lambda)$ which
satisfy \eqref{negdrift_intro}. Fix $0 < \delta < 1/2$. Given two probability measures $\tau_1$ and $\tau_{-1}$ on $\Xi$, we define
\[
\tau(x,\varpi^{n})
=
\left\{
\begin{array}{cl}
\tau_1(x)\delta & \textrm{if $n = 1$} \\[0.2cm]
\tau_{-1}(x)(1-\delta) & \textrm{if $n = -1$} \\[0.2cm]
0 & \textrm{otherwise}
\end{array}
\right.,
\quad 
\textrm{for $(x,\varpi^{n}) \in \Gamma$}.
\]
We note that $\tau$ is a probability measure on $\Gamma$ for which \eqref{negdrift_intro} holds. It thus remains to produce
finitely supported probability measures $\tau_{1}$ and $\tau_{-1}$ on $\Xi$ so that $\tau$ is spread-out and $\Lambda$-absorbing. We first observe that
\[
\alpha_*\tau((x,\varpi^{n})\Lambda)
=
\left\{
\begin{array}{cl}
\Big(\sum_{\xi \in \Xi_o}\tau_1(x + \varpi \xi) \Big) \, \delta & \textrm{if $n = 1$} \\[0.2cm]
\Big(\sum_{\xi \in \Xi_o} \tau_{-1}(x + \varpi^{-1}\xi) \Big) \, (1-\delta) & \textrm{if $n = -1$} \\[0.2cm]
0 & \textrm{otherwise}
\end{array}
\right.,
\quad 
\textrm{for $(x,\varpi^n)\Lambda \in \Gamma/\Lambda$}.
\]
The third expression is clearly $\Lambda$-invariant, and since $\varpi^{-1}\Xi_o \supset \Xi_o$, the middle expression is 
$\Lambda$-invariant as well, for any choice of $\tau_{-1}$. Hence we must only construct $\tau_1$ so that the first expression
is $\Lambda$-invariant. To do this, we pick a set of representatives $T$ for $\Xi_o/\varpi \Xi_o$. Then, for every probability measure
$\kappa$ on $\Xi$, we set
\[
\tau_1(x) = \frac{1}{|T|} \sum_{t \in T} \kappa(x + t), \quad \textrm{for $x \in \Xi$},
\]
and note that $\tau_1$ is a probability measure on $\Xi$, and
\[
\sum_{\xi \in \Xi_o}\tau_1(x + \varpi \xi) = \frac{1}{|T|} \sum_{t \in T} \sum_{\xi \in \varpi \Xi_o} \kappa(x + t + \xi) = \frac{1}{|T|} \sum_{\xi \in \Xi_o} \kappa(x + \xi),
\]
which is clearly $\Xi_o$-invariant, whence $\tau$ with this choice of $\tau_1$ will be $\Lambda$-absorbing. If we further ensure that 
$\tau_{-1}$ and $\kappa$ are finitely supported, so that the resulting support of $\tau$ contains a finite generating set for $\Gamma$, then we have constructed a finitely supported, spread-out and $\Lambda$-absorbing probability measure on $\Gamma$ which satisfies \eqref{negdrift_intro}. \\
\end{remark}

The question arises whether $(K,\nu)$ in Theorem \ref{thm3} is the Poisson boundary of $(\Gamma,\tau)$. Since we are not aware of a good 
reference in this complete generality, we confine our attention to two important special cases, to which the work \cite{Kai0} of Kaimanovich applies.  

\begin{example}[Solvable Baumslag-Solitar groups]
\label{ex1.1}
Let $q$ be a prime number, and set 
\[
K = \bQ_q \qand S = \{0,1\ldots,q-1\} \qand \varpi = q.
\]
Then $\Xi \cong \bZ[1/q]$ and $\Xi_o \cong \bZ$. It is not difficult to see that $\Gamma$ is isomorphic to the Baumslag-Solitar group 
$\BS(1,q) = \langle a,b \, \mid \, bab^{-1} = a^q \rangle$, with $\Lambda \cong \langle a \rangle \cong \bZ$.
\end{example}

\begin{example}[Lamplighter groups]
\label{ex1.2}
Let $q$ be a prime number, and set 
\[
K = \bF_q(t) \qand S = \{0,1\ldots,q-1\} \qand \varpi = t,
\]
where $\bF_q$ denotes the finite field with $q$ elements. Then $\Xi \cong \bigoplus_{\bZ} \bF_q$ and $\Xi_o \cong \bigoplus_{\bN} \bF_q$. It is not difficult to see that $\Gamma$ is isomorphic to the 
wreath product (lamplighter group) $\bF_q \wr \bZ$. Note that in this case, $\Lambda$ is an infinite locally finite group, whence infinitely 
generated.
\end{example}

Kaimanovich \cite[Section 5 and 6]{Kai0} (see also Brofferio \cite{Br0,Br}) has shown that in the two examples above, $(K,\nu)$ is in fact the  
Poisson boundary of $(\Gamma,\tau)$, if $\tau$ is finitely supported and satisfies \eqref{negdrift_intro}. In particular, we now have the following 
corollary. 

\begin{corollary}
\label{cor1_thm3}
Let $\Gamma$ be as in either Example \ref{ex1.1} or Example \ref{ex1.2}, and let $\tau$ be as in Theorem \ref{thm3}.
Then $(\Gamma,\tau)$ is a prime measured group, whose Poisson boundary is not $L^p$-irreducible for any $p \in [1,\infty]$.
\end{corollary}

\subsection{Explicit realizations of boundary entropy spectra}

Kakeya suggested in \cite{Kak} an interesting way to generate closed subsets of the real line along the following lines. Given a positive summable 
sequence $\beta = (\beta_1,\beta_2,\ldots)$, define its \emph{subsum set} $\SubSum(\beta)$ by
\[
\SubSum(\beta) = \Big\{ \sum_{n \in S} \beta_n \, \mid \, S \subseteq \bN \Big\} \subset [0,\infty).
\]
It is not difficult to show that $\SubSum(\beta)$ is always a closed set, and Kakeya proved in 1915 that it is also perfect. After subsequent independent 
work by Hornich \cite{Horn} and Guthrie and Nymann \cite{GuNy} (see also the survey \cite{Ni}), the structure of subsum sets is now very well understood. We summarize 
their findings in the following theorem.

\begin{theorem}
\label{thm_GuNy}
Let $\beta = (\beta_1,\beta_2,\ldots)$ be a positive summable sequence. 
\begin{itemize}
\item[\textsc{(i)}] $\SubSum(\beta)$ is either: \vspace{0.1cm}
\begin{itemize}
\item a finite union of disjoint closed intervals. \vspace{0.1cm}
\item a Cantor set. \vspace{0.1cm}
\item a symmetric Cantorval, i.e. a non-empty compact subset of $[0,\infty)$, which is equal to the closure of its interior, and which has the
property that every pair of endpoints of a non-trivial connected component, consists of accumulation points of one point components.\vspace{0.1cm}
\end{itemize}
\item[\textsc{(ii)}] Suppose that $\beta$ is non-increasing, and set $B_n = \sum_{k > n} \beta_k$ for $n \geq 0$. \vspace{0.1cm}
\begin{itemize}
\item If $\beta_n > B_n$ for all $n \geq 1$, then $\SubSum(\beta)$ is a Cantor set with Lebesgue measure $\lim_n 2^n B_n$. \vspace{0.1cm}
\item If $\beta_n \leq B_n$ for all $n \geq 1$, then $\SubSum(\beta)$ is the interval $[0,B_0]$.
\end{itemize}
\end{itemize}
\end{theorem}

\begin{remark}
In particular, if we let $\beta_n = a\rho^{n-1}$ for some $a > 0$ and $\rho \in (0,1)$, then 
\[
\SubSum(\beta)
=
\left\{
\begin{array}{cc}
\textrm{a Cantor set of Lebesgue measure zero if $\rho < 1/2$}. \\[0.2cm]
\textrm{the interval $[0,\frac{a}{1-\rho}]$ if $\rho \geq 1/2$}.
\end{array}
\right.
\]
\end{remark}

Our next theorem connects subsum sets with boundary entropy spectra of measured groups. We recall from the discussion above 
that the boundary entropy spectrum of a measured group $(\Gamma,\tau)$ is given by
\[
\BndEnt(\Gamma,\tau) = \{ h_\tau(X,\xi) \, \mid \, \textrm{$(X,\xi)$ is a $\tau$-proximal Borel $(\Gamma,\tau)$-space} \big\}.
\]

\begin{theorem}
\label{thm4}
There exists a countable discrete group $\Gamma$ with the following property: for every positive summable 
sequence $\beta = (\beta_1,\beta_2,\ldots)$, there is a spread-out probability measure $\tau_\beta$ on $\Gamma$ such that
\[
\BndEnt(\Gamma,\tau_\beta) = \SubSum(\beta).
\]
In particular we can find spread-out probability measures $\tau$ and $\tau'$ on $\Gamma$ such that the boundary entropy
spectrum of $(\Gamma,\tau)$ is a Cantor set, while the boundary entropy spectrum of $(\Gamma,\tau')$ is an interval.
\end{theorem}

\begin{remark}
It follows from the proof of Theorem \ref{thm4}, in combination with Corollary \ref{cor1_thm3}, that we can for instance take $\Gamma = \bigoplus_{\bN} \BS(1,2)$, the direct sum of countably many copies of the Baumslag-Solitar group $\BS(1,2)$. We do not know if $\Gamma$ in Theorem \ref{thm4}
can be chosen \emph{finitely generated}.
\end{remark}

\subsection{Some questions}
\label{subsec:somequestions}

In this subsection we briefly collect some questions and open problems which are related to the theorems, discussions and corollaries above. \\

In what follows, let $(\Gamma,\Lambda)$ be a Hecke pair, $(H,L,\rho)$ a completion triple of $(\Gamma,\Lambda)$ and $\tau$ a $\Lambda$-absorbing and spread-out probability measure on $\Gamma$. Let $(H,\theta_\tau)$ denote the Hecke completion of $(\Gamma,\tau)$ with respect 
to $\rho$, and write $(B_\tau,\nu_\tau)$ and $(B_{\theta_\tau},\nu_{\theta_\tau})$ for the Poisson boundaries of the measured groups $(\Gamma,\tau)$ and $(H,\theta_\tau)$
respectively.

\vspace{0.2cm}

\noindent \textsc{Martin boundaries} 
\vspace{0.1cm}

\begin{question}
Is there a natural relation between the (minimal) Martin boundary of $(\Gamma,\tau)$ and the (minimal) Martin boundary of $(H,\theta_\tau)$?
\end{question}

\noindent \textsc{A relative Liouville theorem for Hecke pairs} 
\vspace{0.2cm}

The following question asks for a converse to \textsc{(iii)} in Corollary \ref{cor1_Thm2}.

\begin{question}
Suppose that $\Lambda$ is amenable. Given a bi-$L$-invariant spread-out probability measure $\theta$ on $H$, can we find a spread-out $\Lambda$-absorbing probability measure $\tau$ on $\Gamma$ such that $\theta_\tau = \theta$ and the Poisson boundaries of $(\Gamma,\Lambda)$ and $(H,\theta)$ are measurably $\Gamma$-isomorphic?
\end{question}

If $\Lambda$ is a \emph{normal} subgroup of $\Gamma$ (in which case \emph{every} measure on $\Gamma$ is $\Lambda$-absorbing), Kaimanovich \cite[Theorem 1]{Kai1} has answered this question in the affirmative. 
The special case when $\Gamma = \Lambda$ was an influential conjecture of Furstenberg, and was solved independently by Kaimanovich and Vershik in \cite{KV} and by Rosenblatt in \cite{Ro}.

\vspace{0.2cm}

\noindent \textsc{When is $(B_{\theta_\tau},\nu_{\theta_\tau})$ a transitive Borel $H$-space?}
\vspace{0.2cm}

In Subsection \ref{subsec:transitive} (cf. Corollary \ref{cor_densesubgroupfactorhomo}) we show that the study of $\tau$-boundaries simplifies 
significantly if one knows that the $H$-action on $B_{\theta_\tau}$ is ($\nu_{\theta_\tau}$-essentially) transitive. Various criteria for when the 
Poisson boundary of a measured group is transitive was developed by Azencott \cite{Az} for semisimple Lie groups, and by Jaworski \cite{J1,J2}
for almost connected locally compact groups. On the other hand, we are only aware of two general criteria which (also) cover \emph{totally disconnected}
groups. We briefly summarize these criteria:
\vspace{0.1cm}
\begin{itemize}
\item \textbf{$(H,L)$ is a Gelfand pair} (the convolution algebra of compactly supported bi-$L$-invariant functions on $H$ is commutative). In this case, 
the compact group $L$ acts ergodically on $(B_{\theta_\tau},\nu_{\theta_\tau})$, and thus transitively (by compactness of $L$), whence $H$ acts transitively as well (The fact that $L$ acts ergocially follows from the arguments of Monod in \cite{Mo}, see also \cite[Chaper XII]{GuJiTa} for a related approach). An illustrative example of a Gelfand pair is
\[
H = \SL_2(\bQ_p) \qand L = \SL_2(\bZ_p), \quad \textrm{for a prime number $p$}.
\] 
We stress that it is not important in this criterion that $H$ is totally disconnected. \vspace{0.1cm}

\item \textbf{\textbf{$H$ fixes one of infinitely many ends}}. Suppose that $H$ is compactly generated, and that for a fixed compact generating set 
$S \subset H$, the Schreier graph $\cX = \cX(H,L,S)$ associated with triple $(H,L,S)$ has infinitely many ends (the vertices of this graph are $H/L$, and two vertices $xL$ and $yL$ are connected by an edge if the intersection $Lx^{-1}yL \cap S$ is non-empty). Let $\partial \cX$ denote the space of ends of 
$\cX$, and suppose that $H$ fixes an end $\omega$. Then, by \cite{Moll}, $H$ is non-unimodular, amenable and acts transitively on 
$\partial \cX \setminus \{\omega\}$. Furthermore, if $\theta_\tau$ is compactly supported and spread-out, and
\[
\int_H \log \Delta_H \, d\theta_\tau > 0,
\] 
where $\Delta_H$ denotes the modular function on $H$, then \cite[Theorem 6.12b)]{KaWo} asserts that there exists a (unique) $\theta_\tau$-stationary probability measure $\nu$ on $\partial \cX \setminus \{\omega\}$ such that $(\partial \cX \setminus \{\omega\},\nu)$ is the Poisson boundary of 
$(H,\theta_\tau)$.
\end{itemize}

\begin{question}
Are there other criteria for when $(B_{\theta_\tau},\nu_{\theta_\tau})$ is a transitive Borel $H$-space? 
\end{question}

\vspace{0.1cm}
\noindent \textsc{Non-amenable Baumslag-Solitar groups}\\

The Baumslag-Solitar group $\BS(p,q)$ is defined by
\[
\BS(p, q) = \langle a, b \, | \, ab^p a^{-1} = b^q \rangle, \quad \textrm{for non-zero integers $p,q$}.
\]
One readily checks that $(\BS(p,q),\langle b \rangle)$ is a Hecke pair, and an explicit completion triple $(H_{p,q},L_{p,q},\rho_{p,q})$ 
of this Hecke pair was first produced by Gal and Januszkiewicz in \cite{GJ}. \\

Let us fix integers $p$ and $q$ with $2 \leq |p| < |q|$, so that $\BS(p,q)$ is non-amenable, and let 
$\tau$ be a finitely supported $\langle b \rangle$-absorbing probability measure on $\BS(p,q)$. Under mild assumptions, Cuno and 
Sava-Huss \cite{CS} have proved that the space of ends $\Omega_{T_{p,q}}$ of the Bass-Serre tree $T_{p,q}$ associated with $(\BS(p,q),\langle b \rangle)$ is a compact and uniquely $\tau$-stationary model for the Poisson boundary of $(\BS(p,q),\tau)$. Let $\nu$ denote the unique $\tau$-stationary 
probabiluty measure on $\Omega_{T_{p,q}}$. Since $\langle b \rangle \cong \bZ$ is  amenable, \textsc{(IV)} in Corollary \ref{cor1_Thm2} tells us that the Poisson boundary of $(H_{p,q},\theta_{\tau})$ is measurably $\Gamma$-isomorphic to $(\Omega_{T_{p,q}},\nu)$, and thus every $\tau$-boundary is 
measurably $\Gamma$-isomorphic to a $\theta_\tau$-boundary, viewed as a $\Gamma$-space. 

\begin{question}
Is there a "reasonable" parameterization of the set of $\theta_\tau$-boundaries in this setting? 
\end{question}

\vspace{0.1cm}

\noindent \textsc{Anantharaman-Delaroche groups}\\

In the recent paper \cite{AD1}, Anantharaman-Delaroche constructs Hecke pairs $(\Gamma,\Lambda)$ and completion triples $(H,L,\rho)$ 
such that
\vspace{0.1cm}
\begin{itemize}
\item[(i)] $\Lambda$, and thus $\Gamma$, are non-amenable finitely generated groups. \vspace{0.1cm}
\item[(ii)] $\Lambda$ is co-amenable (and thus $H$ is amenable), compactly generated and admits a transitive action on a tree (with compact stabilizers). \vspace{0.1cm}
\item[(iii)] $\rho$ is injective.\vspace{0.1cm}
\end{itemize}
Let us fix such a Hecke pair $(\Gamma,\Lambda)$ and an associated completion triple $(H,L,\rho)$, and let $\theta$ be a \emph{symmetric}, bi-$L$-invariant, 
compactly generated and spred-out probability measure on $H$. By Theorem \ref{thm2}, there exists at least one $\Lambda$-absorbing probability measure $\tau$ on $\Gamma$ (which does not need to be symmetric) such that $\theta = \theta_\tau$. Since $\Gamma$ is non-amenable, the 
Poisson boundary $(B_\tau,\nu_\tau)$ must be non-trivial. However, by \cite[Theorem 6.12a)]{KaWo}, the symmetricity of $\theta_\tau$ forces the Poisson boundary of $(H,\theta_\tau)$ to be trivial. 

\begin{question}
Is $(B_\tau,\nu_\tau)$ a prime $(\Gamma,\tau)$-space (cf. Corollary \ref{cor3_thm2})?
\end{question}

\vspace{0.1cm}

%Suppose that the $H$-action on $(B_{\theta_\tau},\nu_{\theta_\tau})$ is (essentially) transitive, so that we can identify $B_{\theta_\tau}$ (modulo null sets), 
%with a quotient space of the form $(H/P,\nu_P)$, where $P$ is a closed subgroup of $H$, and $\nu_P$ is absolutely continuous with respect to the unique $H$-invariant \emph{measure class} on $H/P$ (see subsection 
%\ref{subsec:transitive} below for more details). By Corollary \ref{cor_densesubgroupfactorhomo}, every $\Gamma$-factor of $(H/P,\nu_o)$ is measurably
%$\Gamma$-isomorphic to $(H/Q,\nu_Q)$, where $Q$ is a closed subgroup of $H$ which contains $P$, and $\nu_Q$ is absolutely continuous with respect
%to the unique $H$-invariant measure class on $H/Q$. These observations tell us that in order to study \emph{$\tau$-boundaries} in the setting at hand (when $(B_{\theta_\tau},\nu_{\theta_\tau})$ is transitive) one
%needs to: 
%\vspace{0.1cm}
%\begin{itemize}
%\item understand when $(B_\tau,\nu_\tau)$ and $(B_{\theta_\tau},\nu_{\theta_\tau})$ are measurably $\Gamma$-isomorphic. Note that amenability of 
%$\Lambda$ is a \emph{necessary} condition for this to hold by \textsc{(III)} in Corollary \ref{cor1_Thm2}, and \textsc{(IV)} in Corollary \ref{cor1_Thm2}
%provides a \emph{sufficient} criterion for when the two spaces are measurably $\Gamma$-isomorphic. \vspace{0.2cm}
%
%\item understand the set of closed subgroups $Q$ of $H$ such that $P <  Q < H$. In particular, if $Q$ is a maximal closed proper subgroup of $H$,
%then $(B_{\theta_\tau},\nu_{\theta_\tau})$ is a prime $(H,\theta_\tau)$-space.
%\end{itemize}

\noindent \textsc{Higman-Thompson's group and the Neretin group} \\

We refer to \cite{CM} for definitions. Let $d \geq 2$ and $k \geq 1$, and denote by $\cT_{d,k}$ the unique rooted tree having $k$ vertices of 
level $1$, each of which is attached to an underlying $d$-regular tree. We write $\Gamma_{d,k}$ and $H_{d,k}$ for the Higman-Thompson group 
and the Neretin group associated with $\cT_{d,k}$ respectively. It is known that there is a injective homomorphism $\rho$ from $\Gamma_{d,k}$ into 
$H_{d,k}$ with a dense image, as well as a compact and open subgroup $L_{d,k}$ such that $\Lambda_{d,k} := \rho^{-1}(L_{d,k})$ is a locally finite 
(and thus amenable) subgroup of $\Gamma_{d,k}$. The following two questions seem to arise naturally.

\begin{question}
Suppose that $\tau$ is a finitely supported and spread-out $\Lambda_{d,k}$-absorbing probability measure on $\Gamma_{d,k}$. 
Does the Poisson boundary of $(\Gamma_{d,k},\tau)$ admit a uniquely $\tau$-stationary compact model (cf. \textsc{(IV)} in Corollary \ref{cor1_Thm2})?
\end{question}

\begin{question}
Is the Poisson boundary of $(H_{d,k},\theta_\tau)$ a transitive $H_{d,k}$-space?
\end{question}

\subsection{Acknowledgements}

This paper is part of H.O's doctoral thesis at Chalmers, under the supervision of M.B. Our collaboration was initiated during Y.H's visit to Chalmers in late August and early September of 2016. The authors are grateful to GoCas (Gothenburg's Center of Advanced Studies) for generously supporting this visit. The collaboration continued during M.B's and H.O's visits to Northwestern University in March and April 2017, as well as during H.O's visit to Ben Gurion University in Israel between January 2019 and August 2019. The authors are very grateful for the hospitality shown to us by these universities.

\section{Proof of Theorem \ref{thm1}}
\label{sec:proofThm1}

Let $\Gamma$ be a countable group $\Lambda$ a subgroup of $\Gamma$, and write $\alpha : \Gamma \ra \Gamma/\Lambda$ for the canonical quotient map. 
The induced map $\alpha_* : \Prob(\Gamma) \ra \Prob(\Gamma/\Lambda)$ is given by
\begin{equation}
\label{def_alphainduced}
\overline{\tau}(\gamma \Lambda) := \alpha_*\tau(\gamma \Lambda) = \sum_{\lambda \in \Lambda} \tau(\gamma \lambda), \quad \textrm{for $\gamma \Lambda \in \Gamma/\Lambda$}.
\end{equation}
We set $\Prob(\Gamma ; \Lambda) = \{ \tau \in \Prob(\Gamma) \, \mid \, \textrm{$\overline{\tau}$ is $\Lambda$-invariant}  \}$.

\subsection*{Proof of \textsc{(I)}}

Fix $\tau_1, \tau_2 \in \Prob(\Gamma ; \Lambda)$, and note that 
\vspace{0.1cm}
\begin{eqnarray}
\alpha_*(\tau_1 * \tau_2)(\gamma \Lambda) 
&=&
\sum_{\lambda \in \Lambda} (\tau_1 * \tau_2)(\gamma \lambda) = \sum_{\eta \in \Gamma} \sum_{\lambda \in \Lambda} \tau_1(\eta) \, \tau_2(\eta^{-1}\gamma \lambda) \nonumber \\[0.2cm]
&=&
\sum_{\eta \in \Gamma} \tau_1(\eta) \, \overline{\tau}_2(\eta^{-1}\gamma \Lambda) 
=
\sum_{\eta \Lambda} \sum_{\lambda' \in \Lambda}\tau_1(\eta \lambda') \, \overline{\tau}_2(\eta^{-1} \gamma \Lambda) \nonumber \\[0.2cm]
&=&
\sum_{\eta \Lambda} \overline{\tau}_1(\eta \Lambda) \, \overline{\tau}_2(\eta^{-1} \gamma \Lambda), \quad \textrm{for $\gamma \Lambda \in \Gamma/\Lambda$}, \label{leftinv}
\end{eqnarray}
where we in the second to last step have used that $\overline{\tau}_2$ is $\Lambda$-invariant. Since $\overline{\tau}_1$ is $\Lambda$-invariant, 
we conclude that $\alpha_*(\tau_1 * \tau_2)$ is $\Lambda$-invariant. In particular, $\tau_1 * \tau_2 \in \Prob(\Gamma ; \Lambda)$ as well.

\subsection*{Proof of \textsc{(II)}}

Suppose that $\Gamma$ acts by linear isometries on a Banach space $E$ and that the dual action  preserves a weak*-closed convex
subset $C$ of $E^*$. If $C^\Lambda$ is empty, there is nothing to prove, so let us assume that $C^\Lambda \neq \emptyset$, and pick $c \in C^\Lambda$.
Then, for every $\tau \in \Prob(\Gamma)$,
\begin{eqnarray*}
\tau * c 
&=&
\sum_{\gamma \in \Gamma} \tau(\gamma) \gamma c 
= 
\sum_{\gamma \Lambda} \sum_{\lambda \in \Lambda} \tau(\gamma \lambda) \gamma \lambda c
=
\sum_{\gamma \Lambda} \overline{\tau}(\gamma \Lambda) \gamma c.
\end{eqnarray*}
If $\tau$ is $\Lambda$-absorbing, then for every $\lambda \in \Lambda$
\begin{eqnarray*}
\lambda(\tau * c) 
&=& 
\sum_{\gamma \Lambda} \overline{\tau}(\gamma \Lambda) \lambda \gamma c
= 
\sum_{\gamma \Lambda} \overline{\tau}(\lambda \gamma \Lambda) \lambda \gamma c \\
&=& 
\sum_{\gamma \Lambda} \overline{\tau}(\gamma \Lambda) \gamma c = \tau * c,
\end{eqnarray*}
whence $\tau * C^\Lambda \subseteq C^\Lambda$. 

\subsection*{Proof of \textsc{(III)}}

Fix a right-inverse $\beta : \Gamma/\Lambda \ra \Gamma$ for the map $\alpha$. Then, for every $\gamma \in \Gamma$, we can write
\[
\gamma = \beta(\gamma \Lambda) \lambda_\gamma,
\]
for some \emph{unique} $\lambda_\gamma \in \Lambda$. In particular, given a probability measure $r$ on $\Lambda$, we can define an (affine)
right-inverse 
\[
\Prob(\Gamma/\Lambda) \ra \Prob(\Gamma), \enskip \widehat{\tau} \mapsto \widetilde{\tau}
\]
of the map $\alpha_*$ by setting
\[
\widetilde{\tau}(\gamma) = \widehat{\tau}(\gamma \Lambda) r(\lambda_\gamma), \quad \textrm{for $\gamma \in \Gamma$}.
\]
Note that
\[
\supp \widetilde{\tau} = \{\gamma \in \Gamma \, \mid \, \alpha(\gamma) \in \supp \widehat{\tau}, \enskip \lambda_\gamma \in \supp r \big\}.
\]
Let us now assume that $(\Gamma,\Lambda)$ is a Hecke pair. Then, since all $\Lambda$-orbits in $\Gamma/\Lambda$ are finite, we get a natural affine 
retraction
\[
\Prob(\Gamma/\Lambda) \ra \Prob(\Gamma/\Lambda)^\Lambda, \enskip \overline{\tau} \mapsto \widehat{\tau}
\]
upon averaging over the $\Lambda$-orbits. Clearly, $\supp \widehat{\tau} = \Lambda.\supp(\overline{\tau})$. Consider the (affine) composition 
\[
\Prob(\Gamma) \ra \Prob(\Gamma ; \Lambda), \enskip \tau \mapsto \widetilde{\tau},
\]
given by $\tau \mapsto \overline{\tau} := \alpha_*\tau \mapsto \widehat{\tau} \mapsto \widetilde{\tau}$. Then, since the support of $\overline{\tau}$
is $\alpha(\supp \tau)$, we see that
\begin{align}
\supp \widetilde{\tau} 
&=
\big\{ \gamma \in \Gamma \, \mid \, \gamma \Lambda \in \Lambda.\alpha(\supp \tau), \enskip \lambda_\gamma \in \supp r \big\} \label{eqsupp} \\[0.2cm]
&\subseteq 
\beta(\Lambda.\alpha(\supp \tau)) \supp r \label{inclsupp}
\end{align}
Fix a subset $S$ of $\Gamma$, and define 
\[
S_\Lambda = \{ \lambda_\gamma \in \Lambda \, \mid \, \gamma \in S \big\}. 
\]
Let us assume that 
the support of the probability measure $r$ above equals $S_\Lambda$, and set 
\[
\widetilde{S} = \beta(\Lambda.\alpha(S))S_\Lambda \subset \Gamma.
\]
If $S$ is finite, then $S_\Lambda$ and $\widetilde{S}$ are finite sets.
It follows from \eqref{inclsupp} that the map $\tau \mapsto \widetilde{\tau}$ restricts to an affine map
\[
\Prob(S) \ra \Prob(\Gamma;\Lambda) \cap \Prob(\widetilde{S}).
\]
Furthermore, if $\gamma \in \supp \tau \subset S$, then clearly 
\[
\gamma \Lambda \in \alpha(\supp \tau) \subset \Lambda.\alpha(\supp \tau) \qand \lambda_\gamma \in S_\Lambda,
\]
whence $\gamma \in \supp \widetilde{\tau}$ by \eqref{eqsupp} and since $r$ has full support on $S_\Lambda$. Since $\gamma$ is arbitrary, we conclude that $\supp \tau \subseteq \supp \widetilde{\tau}$.

\begin{remark}
If we assume that for every $\gamma \in \Gamma$, all elements of $\Lambda$ and $\gamma \Lambda \gamma^{-1}$ commute with each other (this is 
for instance the case for the Hecke pairs in Theorem \ref{thm3}), then there is another affine construction of 
$\Lambda$-absorbing probability measures on $\Gamma$ which can be constructed along the following lines. For every $\gamma \in \Gamma$, we choose sets of representatives 
$A_{\gamma \Lambda}$ and $B_{\gamma \Lambda}$ for the right- and left-quotients 
\[
\Lambda/\Lambda \cap \gamma \Lambda \gamma^{-1} \qand \Lambda \cap \gamma \Lambda \gamma^{-1} \backslash \gamma \Lambda \gamma^{-1}  
\]
respectively. Since all elements of $\Lambda$ and $\gamma \Lambda \gamma^{-1}$ are assumed to commute with each other, and thus
\[
\lambda \gamma \Lambda \gamma^{-1} \lambda^{-1} = \gamma \Lambda \gamma^{-1}, \quad \textrm{for all $\gamma \in \Gamma$ and $\lambda \in \Lambda$},
\]
we may assume that the maps $\gamma \mapsto A_{\gamma \Lambda}$ and $\gamma \mapsto B_{\gamma \Lambda}$ are left-$\Lambda$-invariant. \\

Given $\tau \in \Prob(\Gamma)$, we set
\[
\widetilde{\tau}(\gamma) = 
\frac{1}{|A_{\gamma \Lambda}|} \sum_{a \in A_{\gamma \Lambda}} \tau(a\gamma), \quad \textrm{for $\gamma \in \Gamma$},
\]
and note that $\tau \mapsto \widetilde{\tau}$ is affine, and 
\begin{eqnarray*}
\alpha_*\widetilde{\tau}(\gamma \Lambda) = \sum_{\lambda \in \Lambda} \widetilde{\tau}(\gamma \lambda) 
&=& 
\frac{1}{|A_{\gamma \Lambda}|} \sum_{a \in A_{\gamma \Lambda}} \sum_{\lambda \in \Lambda} \tau(a\gamma \lambda)
=
\frac{1}{|A_{\gamma \Lambda}|} \sum_{a \in A_{\gamma \Lambda}} \sum_{\lambda \in \gamma \Lambda \gamma^{-1}}  
\tau(a \lambda \gamma) \\[0.2cm]
&=& 
\frac{1}{|A_{\gamma \Lambda}|} \sum_{b \in B_{\gamma \Lambda}} \sum_{a \in A_{\gamma \Lambda}} 
\sum_{\lambda \in \Lambda \cap \gamma \Lambda \gamma^{-1}}  \tau(a \lambda b \gamma) \\[0.2cm]
&=&
\frac{1}{|A_{\gamma \Lambda}|} \sum_{b \in B_{\gamma \Lambda}} 
\sum_{\lambda \in \Lambda}  \tau(\lambda b \gamma).
\end{eqnarray*}
Since $B_{\gamma \Lambda} \subset \gamma \Lambda \gamma^{-1}$ and all elements of $\Lambda$ and $\gamma \Lambda \gamma^{-1}$
commute with each other, we have
\[
\alpha_*\widetilde{\tau}(\gamma \Lambda) =\frac{1}{|A_{\gamma \Lambda}|} \sum_{b \in B_{\gamma \Lambda}} 
\sum_{\lambda \in \Lambda}  \tau(b \lambda \gamma), \quad \textrm{for all $\gamma \Lambda \in \Gamma/\Lambda$},
\]
which is clearly a left-$\Lambda$-invariant expression since the maps $\gamma \mapsto A_{\gamma \Lambda}$ and $\gamma \mapsto B_{\gamma \Lambda}$ are left-$\Lambda$-invariant. 
\end{remark}

%%%%%%%%%%%%%%%%%%%%%%%%%%%%%%%%%%%%%%%%%%%%%%%%%%%%%%%%%%%%%%%%%%%%%%%%%

\section{Preliminaries on compact $G$-spaces with $\mu$-stationary measures}
\label{sec:prelcpt}
Let $G$ be a locally compact and second countable (lcsc) group and let $\mu$ be a probability measure on $G$. 
Let $M$ be a compact and metrizable space, and denote by $\Prob(M)$ the space of Borel probability measures on $M$, 
equipped with the (compact and metrizable) weak*-topology. If $G$ acts jointly continuously by homeomorphisms
on $M$, then we say that $M$ is a \emph{compact $G$-space}. In this case, $G$ also acts jointly continuously by homemorphisms on
$\Prob(M)$, and we say that $\eta \in \Prob(M)$ is \emph{$\mu$-stationary} if $\mu * \eta = \eta$. 

We denote by
$\Prob_\mu(M)$ the set of all $\mu$-stationary probability measures on $M$, and write 
$\Prob_\mu^{\textrm{erg}}(M)$ and $\Prob_\mu^{\textrm{ext}}(M)$ for the subsets of ergodic and extremal measures in
$\Prob_\mu(M)$ respectively.

\begin{lemma}
\label{lemma_basicprops}
Let $\mu$ be a spread-out probability measure on $G$ and let $M$ be a compact $G$-space.
\vspace{0.1cm}
\begin{itemize}
\item[\textsc{(i)}] Every $\mu$-stationary probability measure on $M$ is $G$-quasi-invariant. \vspace{0.2cm}
\item[\textsc{(ii)}] $\Prob^{\textrm{erg}}_\mu(M) = \Prob^{\textrm{ext}}_\mu(M) \neq \emptyset$. \vspace{0.2cm}
\item[\textsc{(iii)}] If $\eta$ and $\eta'$ are ergodic $\mu$-stationary probability measures and if $\eta'$ is absolutely 
continuous with respect to $\eta$, then $\eta = \eta'$.
\end{itemize}
\end{lemma}

\begin{proof}
 \textsc{(ii)} is proved in \cite[Lemma 1.1]{NZ}. The identity in \textsc{(I)} is contained in \cite[Corollary 2.7]{BS}, while the assertion of 
 non-emptiness is an immediate consequence of Kakutani's fixed point theorem. \textsc{(III)} is 
proved in \cite[Proposition 2.6:(2)]{BS}.
\end{proof}

\begin{corollary}
\label{cor_mu1mu2}
Let $\mu_1$ and $\mu_2$ be spread-out probability measures on $G$ and let $M$ be a compact $G$-space. 
Suppose that $\mu_1 * \mu_2 = \mu_2 * \mu_1$. Then $\Prob_{\mu_1}(M) = \Prob_{\mu_2}(M)$.
\end{corollary}

\begin{proof}
By \textsc{(II)} in Lemma \ref{lemma_basicprops} it suffices to show that $\Prob^{\textrm{erg}}_{\mu_1}(M) \subset \Prob_{\mu_2}(M)$. To prove this inclusion, 
fix an ergodic $\eta \in \Prob_{\mu_1}(M)$ and set $\eta' := \mu_2 * \eta$. Since $\mu_1$ is spread-out, 
$\eta$ is $G$-quasi-invariant by \textsc{(I)} in Lemma \ref{lemma_basicprops}, and thus $\eta' \ll \eta$. Note that since $\mu_1 * \mu_2 = \mu_2 * \mu_1$, 
\[
\mu_1 * \eta' = \mu_1 * \mu_2 * \eta = \mu_2 * \mu_1 * \eta = \mu_2 * \eta = \eta',
\]
whence $\eta'$ is $\mu_1$-stationary and absolutely continuous with respect to $\eta$, and thus $\eta' = \eta$ by \textsc{(iii)} in Lemma \ref{lemma_basicprops}. This shows that $\eta \in \Prob_{\mu_2}(M)$, and we are done.
\end{proof}

\subsection{Conditional measures and the canonical quasi-factor}
\label{subsec:canonicalquasi}

\begin{lemma} \cite[Theorem 2.10]{BS}
\label{lemma_conditional}
Let $\mu$ be a probability measure on $G$, $M$ a compact $G$-space and $\eta$ a $\mu$-stationary probability measure on $M$.
There is a $\mu^{\bN}$-conull subset $\Omega_\eta \subset G^{\bN}$ and a measurable map 
\[
\Omega \ra \Prob(M), \enskip \omega \mapsto \eta_\omega,
\]
such that $\lim_n \omega_1 \cdots \omega_n \eta = \eta_\omega$, for all $\omega = (\omega_1,\omega_2,\ldots) \in \Omega_\eta$, where the 
limit is taken in the weak*-topology on $\Prob(M)$.
\end{lemma}

\begin{remark}
We stress that we do not assume that $\mu$ is a spread-out probability measure on $G$.
\end{remark}

\begin{definition}[Conditional measures]
The map $\omega \mapsto \eta_\omega$ in Lemma \ref{lemma_conditional} is called the \emph{conditional measure map} associated
with $(M,\eta)$, and the measures $(\eta_\omega)$ are called \emph{conditional measures}. 
\end{definition}

In what follows, let $\mu$ be a probability measure on $G$ and let $(M,\eta)$ be a compact $(G,\mu)$-space. Given a bounded Borel function 
$f$ on $M$, we define 
\[
\widehat{f}(\beta) = \beta(f), \quad \textrm{for $\beta \in \Prob(M)$}.
\]
We note that if $f$ is continuous (Borel measurable), then $\widehat{f}$ is continuous (Borel measurable) with respect to the weak*-topology 
on $\Prob(M)$. Moreover, 
\begin{equation}
\label{quasifactor}
\int_M f \, d\eta = \int_{G^{\bN}} \widehat{f}(\eta_\omega) \, d\mu^{\otimes \bN}(\omega), \quad \textrm{for all $f \in C(M)$},
\end{equation}
where $\omega \mapsto \eta_\omega$ denotes the conditional measure map associated with $(M,\eta)$. A straightforward approximation
argument in $L^1(\eta)$ also shows that \eqref{quasifactor} holds for every bounded Borel function on $M$.

\begin{definition}[The canonical quasi-factor]
The probability measure $\eta^*$ on $\Prob(M)$ defined by
\[
\eta^*(F) = \int_{G^{\bN}} F(\eta_\omega) \, d\mu^{\otimes \bN}(\omega), \quad \textrm{for $F \in C(\Prob(M))$},
\]
where $\omega \mapsto \eta_\omega$ is the conditional measure map associated with $(M,\eta)$, is called the
\emph{canonical quasi-factor} of $(M,\eta)$.
\end{definition}

\begin{remark}
We note that \eqref{quasifactor} says that $\eta^*(\widehat{f}) = \eta(f)$ for all $f \in C(M)$, which in particular implies that $(\Prob(M),\eta^*)$ is 
a $\mu$-stationary quasi-factor of $(M,\eta)$ in the sense of Furstenberg and Glasner (see e.g. \cite[Section 1]{FuGl}).
\end{remark}

It is not hard to see that the set of all $\widehat{f}$, as $f$ ranges over $C(M)$, separates points in $\Prob(M)$, whence the *-algebra 
generated by such functions is dense in $C(\Prob(M))$ by Stone-Weierstrass Theorem. In particular, $\eta^*$, is completely determined by
all expressions of the form $\eta^*(\widehat{f}_1 \cdots \widehat{f}_k)$, where $k \geq 1$ and $f_1,\ldots,f_k$ is a $k$-tuple
in $C(M)$.  The following lemma provides a technique to evaluate these expression. 

\begin{lemma}
\label{lemma_canonicalquasifactor}
For every $k \geq 1$, and for all bounded Borel functions $f_1,\ldots,f_k$ on $M$, 
\[
\eta^*(\widehat{f}_1 \cdots \widehat{f}_k) = \lim_{n \ra \infty} (\mu^{*n} * \eta^{\otimes k})(f_1 \otimes \cdots \otimes f_k).
\]
\end{lemma}

\begin{proof}
Fix $k \geq 1$. Let us first consider the case when $f_1,\ldots,f_k$ are continuous. We note that
\begin{eqnarray*}
(\mu^{*n} * \eta^{\otimes k})(f_1 \otimes \cdots \otimes f_k) 
&=&
\int_{G} g \eta(f_1) \cdots g \eta(f_k) \, d\mu^{*n}(g) \\[0.2cm]
&=&
\int_{G^{\bN}} (\omega_1 \cdots \omega_n \eta)(f_1) \cdots (\omega_1 \cdots \omega_n) \eta(f_k) \, d\mu^{\bN}(\omega).
\end{eqnarray*}
By Lemma \ref{lemma_conditional}, 
\[
\eta_\omega(f) = \lim_{n \ra \infty} (\omega_1 \cdots \omega_n \eta)(f), \quad \textrm{for all $f \in C(M)$},
\]
$\mu^{\bN}$-almost surely, whence by dominated convergence
\[
\lim_{n \ra \infty} 
\int_{G^{\bN}} (\omega_1 \cdots \omega_n \eta)(f_1) \cdots (\omega_1 \cdots \omega_n) \eta(f_k) \, d\mu^{\bN}(\omega)
=
\int_{G^{\bN}} \eta_\omega(f_1) \cdots \eta_\omega(f_k) \, d\mu^{\bN}(\omega).
\]
We now note that
\[
\int_{G^{\bN}} \eta_\omega(f_1) \cdots \eta_\omega(f_k) \, d\mu^{\bN}(\omega) \\[0.2cm]
=
\int_{G^{\bN}} \widehat{f}_1(\eta_\omega) \cdots \widehat{f}_k(\eta_\omega) \, d\mu^{\bN}(\omega) \\[0.2cm]
=
\eta^*(\widehat{f}_1 \cdots \widehat{f}_k),
\]
which proves the lemma for continuous functions. \\

To prove the lemma in general, pick a $k$-tuple $(f_1,\ldots,f_k)$ of bounded Borel functions on $M$. For every $\eps > 0$, we can find 
a $k$-tuple $(f_{1,\eps},\ldots,f_{k,\eps})$ of continuous functions on $M$ such that
\[
\|f_{i,\eps}\|_\infty \leq \|f_{i}\|_\infty \qand \|f_{i} - f_{i,\eps}\|_{L^1(\eta)} < \eps, \quad \textrm{for all $i = 1,\ldots,k$}.
\]
We write
\begin{eqnarray*}
\eta^*(\widehat{f}_1 \cdots \widehat{f}_k)
&=&
\eta^*(\widehat{f}_{1,\eps} \cdots \widehat{f}_{k,\eps}) \\[0.2cm]
&+&
\sum_{j=0}^{k-1} \eta^*\Big( \big( \prod_{i=1}^j \widehat{f}_{i,\eps} \big) \cdot \big(\widehat{f}_{j+1} - \widehat{f}_{j+1,\eps}\big) \cdot \big( \prod_{i=j+2}^k \widehat{f}_{i}\big) \Big),
\end{eqnarray*}
and
\begin{eqnarray*}
(\mu^{*n} * \eta^{\otimes k}\big)(f_1 \otimes \cdots \otimes f_k)
&=&
(\mu^{*n} * \eta^{\otimes k}\big)(f_{1,\eps} \otimes \cdots \otimes f_{k,\eps}) \\[0.2cm]
&+&
\sum_{j=0}^{k-1} (\mu^{*n} * \eta^{\otimes k}\big)\Big( \big( \bigotimes_{i=1}^j f_{i,\eps} \big) \otimes \big(f_{j+1} - f_{j+1,\eps}\big) \otimes \big( \bigotimes_{i=j+2}^k f_{i}\big) \Big),
\end{eqnarray*}
for all $n \geq 1$, with the convention that products or tensor products over empty sets are equal to one. By \eqref{quasifactor}, 
\[
\big| \eta^*(\widehat{f}_{j+1} - \widehat{f}_{j+1,\eps}) \big| \leq \|f_{j+1} - f_{j+1,\eps}\|_{L^1(\eta)} < \eps, 
\]
for all $j=1,\ldots,k-1$, whence 
\[
\big|\eta^*(\widehat{f}_1 \cdots \widehat{f}_k)
-
\eta^*(\widehat{f}_{1,\eps} \cdots \widehat{f}_{k,\eps})
\big|
\leq \eps \sum_{j=1}^k \prod_{i \neq j} \|f_{i,\eps}\|_\infty.
\]
Since $\eta$ is $\mu$-stationary, the marginal measures of $\mu^{*n} * \eta^{\otimes k}$ are all equal to $\eta$, and thus
\[
\big|
(\mu^{*n} * \eta^{\otimes k}\big)(f_1 \otimes \cdots \otimes f_k)
-
(\mu^{*n} * \eta^{\otimes k}\big)(f_{1,\eps} \otimes \cdots \otimes f_{k,\eps})
\big|
\leq
 \eps \sum_{j=1}^k \prod_{i \neq j} \|f_{i,\eps}\|_\infty,
\]
for all $n \geq 1$. Since $\eps > 0$ and $n$ are arbitrary, and since 
\[
\lim_n (\mu^{*n} * \eta^{\otimes k}\big)(f_{1,\eps} \otimes \cdots \otimes f_{k,\eps})
=
\eta^*(\widehat{f}_{1,\eps} \cdots \widehat{f}_{k,\eps}),
\]
(since $f_{1,\eps},\ldots,f_{k,\eps}$ are continuous), we conclude that 
\[
\lim_n (\mu^{*n} * \eta^{\otimes k}\big)(f_{1} \otimes \cdots \otimes f_{k}) = \eta^*(\widehat{f}_1 \cdots \widehat{f}_k).
\]
\end{proof}

\subsection{Proximality and couplings}

Let $G$ be a lcsc group and $\mu$ a probability measure on $G$. A $\mu$-stationary probability measure $\eta$ on $M$
is \emph{$\mu$-proximal} if the conditional measure map $\omega \mapsto \eta_\omega$, restricted to a $\mu^{\bN}$-conull subset of $\Omega_\eta$, takes values in the set of point measures 
on $M$. 

\begin{definition}[Couplings]
Let $k \geq 2$ be an integer. A Borel probability measure $\kappa$ on $M^k$ is a \emph{$k$-coupling} of $(M,\eta)$ if the push-forward
of $\kappa$ to every $M$-factor equals $\eta$. We denote by $\eta_{\Delta_k}$ the $k$-coupling of $(M,\eta)$ which is supported on
the $k$-diagonal in $M^k$. 
\end{definition}

\begin{remark}
Every probability measure $\eta$ on $M$ has two "trivial" $2$-couplings: the product measure $\eta \otimes \eta$ and the $2$-diagonal
measure $\eta_{\Delta_2}$. If $\eta$ is $\mu$-stationary, then $\eta_{\Delta_2}$ is always $\mu$-stationary, but $\eta \otimes \eta$ mostly 
fails to be $\mu$-stationary (unless $\eta$ is $G$-invariant).
\end{remark}

\begin{lemma}
\label{lemma_proxtransfer}
Let $\mu$ be a probability measure on $G$, $M$ a compact $G$-space and $\eta$ a $\mu$-stationary probability measure on $M$. The following conditions are equivalent.
\vspace{0.1cm}
\begin{enumerate}
\item[\textsc{(I)}] $\eta$ is $\mu$-proximal. \vspace{0.2cm}
\item[\textsc{(II)}] $\eta_{\Delta_2}$ is the only $\mu$-stationary $2$-coupling of $(M,\eta)$.  \vspace{0.2cm}
\item[\textsc{(III)}] $\lim_{n} \mu^{*n} * \eta^{\otimes 4} = \eta_{\Delta_4}$. \vspace{0.2cm}
\item[\textsc{(IV)}] $\lim_{n} \mu^{*n} * \eta^{\otimes k} = \eta_{\Delta_k}$, for all $k \geq 2$.
\end{enumerate}
\end{lemma}

\begin{remark}
While the equivalence between \textsc{(I)} and \textsc{(II)} is fairly standard (see e.g. \cite{Fur3} for related results), the equivalence between
\textsc{(I)} and \textsc{(III)} seems to be new (at least we have not been able to find an explicit reference in the literature).
\end{remark}

\begin{proof}
(\textsc{I) $\implies$ \textsc{(II)}}: Suppose that $\eta$ is $\mu$-proximal, and pick a $\mu$-stationary self-coupling $\kappa$ of the 
$(G,\mu)$-space $(M,\eta)$.
Then $\kappa_{\omega}$ is almost surely a self-coupling of $\eta_\omega$, whence equal to $\eta_\omega \otimes \eta_\omega$ since 
$\eta_\omega$ is a point measure. Since 
\[
\eta = \int_{\Omega_\eta} \eta_\omega \, d\mu^{\otimes \bN}(\omega) 
\qand 
\kappa = \int_{\Omega_\eta} \eta_\omega \otimes \eta_\omega \, d\mu^{\otimes \bN}(\omega),
\]
we conclude that $\kappa = \eta_{\Delta_2}$. \\

(\textsc{II) $\implies$ \textsc{(IV)}}: Fix $k \geq 2$, and note that 
\[
\kappa_k := \lim_{n} \mu^{*n} * \eta^{\otimes k} = \lim_n \int_{\Omega_\eta} (z_n(\omega)\eta)^{\otimes k} \, d\mu^{\otimes \bN}(\omega) = 
\int_{\Omega_\eta} \eta_\omega^{\otimes k} \, d\mu^{\otimes \bN}(\omega),
\]
where $z_n(\omega) = \omega_1 \cdots \omega_n$ for $\omega \in \Omega_\eta$, is a $\mu$-stationary $k$-coupling  of $(M,\eta)$ (the limit exists by Lemma \ref{lemma_conditional}). In particular,
$\kappa_2 = \eta_{\Delta_2}$ by \textsc{(ii)}, and if $k > 2$, then the push-forward of $\kappa_k$ to every $M \times M$-factor 
must equal $\eta_{\Delta_2}$.  By induction, we conclude that $\kappa_k = \eta_{\Delta_k}$. \\

(\textsc{IV) $\implies$ \textsc{(III)}}: Trivial. \\

(\textsc{III) $\implies$ \textsc{(I)}}: To prove that $\eta_\omega$ is almost surely a point measure, it suffices to show that there is a 
$\bP_\mu$-conull subset $\Omega \subset G^{\bN}$ such that 
\begin{equation}
\label{pointmeas}
\eta_\omega(f_1 f_2) = \eta_\omega(f_1) \, \eta_\omega(f_2), \quad \textrm{for all $\omega \in \Omega$ and $f_1, f_2 \in C(M)$}.
\end{equation}
Since $M$ is metrizable, $C(M)$ equipped with the uniform norm, is a separable Banach space. Let $(f_i)$ be a countable norm-dense subset of $C(M)$. To prove \eqref{pointmeas}, 
it is clearly enough to show that
\[
\int_{G^{\bN}} \big| \eta_\omega(f_i f_j) - \eta_\omega(f_i) \, \eta_\omega(f_j) \big|^2 \, d\mu^{\otimes \bN}(\omega) = 0, \quad \textrm{for all $i,j$}.
\]
Upon expanding the square, we see that this amounts to proving:
\begin{equation}
\label{tobeproved_nuomega3}
\kappa_2(f_if_j \otimes f_if_j) - 2 \kappa_3(f_i f_j \otimes f_i \otimes f_j) + \kappa_4(f_i \otimes f_i \otimes f_j \otimes f_j) = 0, 
\end{equation}
for all $i$ and $j$, where
\[
\kappa_k = \int_{G^{\bN}} \eta_\omega^{\otimes k} \, d\mu^{\otimes \bN}(\omega), \quad \textrm{for $k = 2,3,4$}.
\]
Since we assume that $\lim_n \mu^{*n} * \eta^{\otimes 4} = \eta_{\Delta_4}$, we have
\begin{eqnarray*}
\kappa_4 
&=& 
\int_{G^{\bN}} \eta_\omega^{\otimes 4} \, d\mu^{\otimes \bN}(\omega) \\
&=& 
\lim_n \int_{G^{\bN}} (z_n(\omega)_*\eta)^{\otimes 4} \, d\mu^{\otimes \bN}(\omega) \\
&=& 
\lim_n \mu^{*n} * \eta^{\otimes 4} = \eta_{\Delta_4},
\end{eqnarray*}
and thus $\kappa_3 = \eta_{\Delta_3}$ and $\kappa_2 = \eta_{\Delta_2}$ as well. We conclude that 
\[
\kappa_2(f_i f_j \otimes f_if_j) = \kappa_3(f_i f_j \otimes f_i \otimes f_j) = \kappa_4(f_i \otimes f_i \otimes f_j \otimes f_j) = \eta(f_i^2 \, f_j^2),
\]
for all $i$ and $j$, from which \eqref{tobeproved_nuomega3} readily follows, and we are done.
\end{proof}

The following corollary is an immediate consequence of the equivalence between the conditions \textsc{(I)} and \textsc{(III)} in the lemma above.

\begin{corollary}
\label{cor_4power}
Let $\mu_1$ and $\mu_2$ be probability measures on $G$, $M$ a compact $G$-space and $\eta$ a probability measure on $M$ which is both 
$\mu_1$-stationary and $\mu_2$-stationary. Suppose that 
\[
\mu_1^{*n} * \eta^{\otimes 4} = \mu_2^{*n} * \eta^{\otimes 4}, \quad  \textrm{for all $n$}. 
\]
Then $\eta$ is $\mu_1$-proximal if and only if it is $\mu_2$-proximal.
\end{corollary}

We stress that we did not assume in the previous corollary that the measures $\mu_1$ and $\mu_2$ are spread-out. However, in the next corollary,
we shall assume that they are. 

\begin{corollary}
\label{cor_PBcommute}
Let $\mu_1$ and $\mu_2$ be spread-out probability measures on $G$, $M$ a compact $G$-space and $\eta$ a probability measure on $M$
which is both $\mu_1$-stationary and $\mu_2$-stationary. Suppose that $\mu_1 * \mu_2 = \mu_2 *\mu_1$. Then $\eta$ is $\mu_1$-proximal if 
and only if it is $\mu_2$-proximal. 
\end{corollary}

\begin{proof}
We shall assume that $\eta$ is $\mu_1$-proximal and prove that $\eta_{\Delta_2}$ is the only 
$\mu_2$-stationary self-coupling of $(M,\eta)$. In view of the equivalence between \textsc{(I)} and \textsc{(II)} in Lemma \ref{lemma_proxtransfer},
this will finish the proof of the corollary. We fix a $\mu_2$-stationary self-joining $\kappa$ of $(M,\eta)$. By Corollary \ref{cor_mu1mu2} applied to the compact $G$-space $M \times M$, $\kappa$ is also $\mu_1$-stationary, and thus equal to $\eta_{\Delta_2}$ since $\eta$ is $\mu_1$-proximal.
\end{proof}

%%%%%%%%%%%%%%%%%%%%%%%%%%%%%%%%%%%%%%%%%%%%%%%%%%

\section{Preliminaries on Borel $(G,\mu)$-spaces}
\label{sec:prelBorel}

Let $G$ be a lcsc group, $\mu$ a probability measure on $G$, and $(X,\xi)$ a Borel $(G,\mu)$-space. We shall always assume that 
our spaces are \emph{standard}. The space $\cH^\infty(G,\mu)$
of bounded \emph{$\mu$-harmonic functions on $G$} is defined by
\[
\cH^\infty(G,\mu) = \{ \varphi \in \cL^\infty(G) \, : \, \varphi * \mu = \varphi \big\}.
\]
If $\mu$ is spread-out, $\cH^\infty(G,\mu) \subset C_b(G)$. We denote by $P_\xi : L^\infty(X,\xi) \ra \cH^\infty(G,\mu)$ the
\emph{Poisson transform of $(X,\xi)$}, which is defined by
\[
P_\xi f(g) = \int_X f(gx) \, d\xi(x), \quad \textrm{for $g \in G$ and $f \in L^\infty(X,\xi)$}.
\]
If $(X',\xi')$ is another Borel $(G,\mu)$-space, and there are $G$-invariant conull subsets $X_o \subset X$ and $X'_o \subset X'$,
and a Borel $G$-map $\pi : X_o \ra X_o'$ such that the measure class of the push-forward measure $\pi_*(\xi \mid_{X_o})$ equals the 
measure class of $\xi'$, we say that $\pi$ is a \emph{measurable $G$-map}. If moreover, $\pi_*(\xi \mid_{X_o})
= \xi'_{X_o'}$, we say that $\pi$ is \emph{measure-preserving}. To make notation less heavy, we shall suppress the dependences on the subsets $X_o$ and $X_o'$, and simply write $\pi : (X,\xi) \ra (X',\xi')$ if $\pi$ is a measurable $G$-map, and then emphasize if this map is measure-preserving or not. 
If $\pi$ is measure-preserving, we also say that $(X',\xi')$ is a \emph{$G$-factor} of $(X,\xi)$, and if $\pi$ admits a measurable inverse, which is a measure-preserving measurable $G$-map, we say that the Borel $(G,\mu)$-spaces $(X,\xi)$ and $(X',\xi')$ are \emph{measurably $G$-isomorphic}.

\subsection{Compact models}

\begin{definition}[Compact model]
Let $(X,\xi)$ be a Borel $(G,\mu)$-space, $M$ a compact $G$-space and $\eta$ a $\mu$-stationary probability measure on $M$. 
If $(M,\eta)$ is measurably $G$-isomorphic to $(X,\xi)$, we say that $(M,\eta)$ is a \emph{compact model} of $(X,\xi)$.
\end{definition}

\begin{proposition}
\label{prop_cptmodel}
Let $(X,\xi)$ be a Borel $(G,\mu)$-space.
\begin{itemize}
\item[\textsc{(i)}] There exists a compact model for $(X,\xi)$. \vspace{0.2cm}
\item[\textsc{(ii)}] If $(X',\xi')$ is a Borel $(G,\mu)$-space, and $\pi : (X,\mu) \ra (X',\xi')$ is a measure-preserving $G$-map, and 
$(M,\eta)$ and $(M',\eta')$ are compact models for $(X,\xi)$ and $(X',\xi')$ respectively, then there is a 
Borel $G$-map $\sigma: M \ra M'$, which intertwines $\pi$, such that
\[
\sigma_*\eta = \eta' \qand \sigma_*\eta_\omega = \eta'_\omega, \quad \textrm{for $\mu^{\otimes \bN}$-almost every $\omega$},
\]
where $\omega \mapsto \eta_\omega$ and $\omega \mapsto \eta'_\omega$ denote the conditional measure maps associated with $\eta$ and $\eta'$
respectively.
\end{itemize}
\end{proposition}

\begin{proof}
\textsc{(i)} and the first part of \textsc{(ii)} are stated, with appropriate references, in \cite[Theorem 2.1]{BS}, while the second part of 
\textsc{(ii)} is proved in \cite[Corollary 2.7]{BS}.
\end{proof}

The key point of Proposition \ref{prop_cptmodel} is that every Borel $(G,\mu)$-space can be endowed with a (measure class) of conditional
measure maps. Indeed, if $(M,\eta)$ is a compact model for this space, then there is a conditional measure map by Lemma \ref{lemma_conditional}, 
and if $(M',\eta')$ is another compact model, then the induced measure-preserving $G$-map between $(M,\eta)$ and $(M',\eta')$ (which is Borel), 
pushes the conditional measure map for $(M,\eta)$ to the conditional measure map for $(M',\eta')$ (modulo $\mu^{\bN}$-null sets) by 
\textsc{(ii)} in Proposition \ref{prop_cptmodel}. 

\subsection{Proximal Borel $(G,\mu)$-spaces}

\begin{definition}[$\mu$-boundary]
We say that a Borel $(G,\mu)$-space $(X,\xi)$ is \emph{$\mu$-proximal} (or is a \emph{$\mu$-boundary}) if every compact model of $(X,\xi)$ is $\mu$-proximal.
\end{definition}

\begin{remark}
By Proposition \ref{prop_cptmodel}, between any two different compact models of $(X,\xi)$, there is a measure-preserving Borel 
$G$-isomorphism, which maps the conditional measures to conditional measures (modulo $\mu^{\bN}$-null sets), so in particular
the notion of $\mu$-proximality for Borel $(G,\mu)$-spaces is well-defined (and restricts to $\mu$-proximality for compact $G$-spaces).
\end{remark}

By passing to compact models in Lemma \ref{lemma_proxtransfer}, Corollary \ref{cor_4power} and Corollary \ref{cor_PBcommute}, we 
can now immediately formulate the following results.

\begin{proposition}
\label{prop_proximalktuples}
Let $G$ be a lcsc group and $\mu$ a probability measure on $G$. Let $(X,\xi)$ be a $\mu$-proximal Borel $(G,\mu)$-space. 
Then, for every $k \geq 2$ and for every $k$-tuple $(f_1,\ldots,f_k)$ of bounded Borel functions on $X$,
\[
\lim_n \, (\mu^{*n} * \xi^{\otimes k})(f_1 \otimes \cdots \otimes f_k) = \xi(f_1 \cdots f_k).
\]
\end{proposition}

\begin{remark}
A word of caution here: we do not claim that upon passing to a compact model of $(X,\xi)$, the functions $f_1,\ldots,f_k$ become
continuous, and Lemma \ref{lemma_proxtransfer} can be applied of-the-shelf. Rather, if $(M,\eta)$ is a $\tau$-proximal compact
model of $(X,\xi)$, we apply \textsc{(IV)} in Lemma \ref{lemma_proxtransfer} to conclude that 
$\lim_n \mu^{*n} * \eta^{\otimes k} = \eta_{\Delta_k}$ in the weak*-sense. Then, just as in the proof of Lemma \ref{lemma_canonicalquasifactor}, 
we conclude that the convergence also holds for arbitrary $k$-tuples of bounded Borel functions on $M$.
\end{remark}

\begin{proposition}
\label{prop_proxcrit}
Let $G$ be a lcsc group, let $\mu_1$ and $\mu_2$ be probability measures on $G$ and let $(X,\xi)$ be a Borel $(G,\mu_1)$-space which is also a Borel 
$(G,\mu_2)$-space. \vspace{0.1cm}
\begin{itemize}
\item[\textsc{(i)}] If $\mu_1^{*n} * \xi^{\otimes 4} = \mu_2^{*n} * \xi^{\otimes 4}$ for all $n$, then $\xi$ is $\mu_1$-proximal if and only if it is $\mu_2$-proximal. \vspace{0.1cm}
\item[\textsc{(ii)}] If $\mu_1$ and $\mu_2$ are spread-out and $\mu_1 * \mu_2 = \mu_2 * \mu_1$, then $\xi$ is $\mu_1$-proximal if and only if it is $\mu_2$-proximal.
\end{itemize}
\end{proposition}

Let us also record here some following well-known properties of the canonical quasi-factor, introduced in Subsection \ref{subsec:canonicalquasi} above.

\begin{proposition}
\label{prop_proximalityandcanonicalquasi}
Let $M$ be a compact $G$-space, let $\eta$ be a $\mu$-stationary probability measure on $M$ and let $(\Prob(M),\eta^*)$ denote the canonical quasi-factor
associated with $(M,\eta)$.
 \vspace{0.1cm}
\begin{itemize}
\item[\textsc{(i)}] $(\Prob(M),\eta^*)$ is a $\mu$-boundary. \vspace{0.1cm}
\item[\textsc{(ii)}] If $\eta$ is $\mu$-proximal, then $(\Prob(M),\eta^*)$ is measurably $G$-isomorphic to $(M,\eta)$.
\end{itemize}
\end{proposition}

\begin{proof}
\textsc{(i)} can for instance be found in \cite[Proposition 3.2]{Fur3}, while \textsc{(ii)} is contained in \cite[Theorem 4.3]{FuGl}.
\end{proof}

\subsection{The Poisson boundary of a measured group}

The following fundamental construction is due to Furstenberg \cite{Fur1}.

\begin{theorem}
\label{thm_Poisson}
For every measured group $(G,\mu)$ there exists a $\mu$-proximal Borel $(G,\mu)$-space $(B,\nu)$ (which is unique modulo measurable $G$-isomorphisms) with the following properties: \vspace{0.1cm}
\begin{myindentpar}{0.2cm}
\textsc{Maximality:} Every $\mu$-boundary is a $G$-factor of $(B,\nu)$. \vspace{0.2cm}
\end{myindentpar}

\begin{myindentpar}{0.2cm}
\textsc{Poisson representation:} The Poisson transform $P_\nu : L^\infty(B,\nu) \ra \cH^\infty(G,\mu)$ is an isometric isomorphism.
\end{myindentpar}
\end{theorem}

\begin{remark}
We stress that we assume in Theorem \ref{thm_Poisson} that $\mu$ is spread-out. 
\end{remark}

\begin{definition}[Poisson boundary]
The Borel $(G,\mu)$-space $(B,\nu)$ (or any of its measurable $G$-isomorphic images) in Theorem \ref{thm_Poisson} is called the 
\emph{Poisson boundary} of the measured group $(G,\mu)$.
\end{definition}

\subsection{Dense subgroups}

We shall now prove one of the key ingredient in the proof of \textsc{(ii)} in Theorem \ref{cor1_Thm2}. Although we have not seen this result 
explicitly stated in the literature, it is a fairly immediate consequence of the fundamental works of Mackey and Ramsey, as outlined in the 
appendix to Zimmer's book \cite{Zi}. For completeness, we shall provide the details . We use the same terminology of measurable $G$-maps 
and $G$-factors as in the beginning of the section, but with the difference that our measures are not necessarily $\mu$-stationary (but only 
quasi-invariant under $G$, and certain dense subgroups thereof). For this reason we shall refer to Borel $G$-spaces, and not Borel $(G,\mu)$-spaces,
and we shall not insist that our maps are measure-preserving, but only that they map measure classes to measure classes. 
As before, we always assume that the Borel spaces involved are standard. \\

In what follows, let
\begin{itemize}
\item $G$ be a lcsc group and $G_o$ a \emph{dense} subgroup of $G$. \vspace{0.1cm}
\item $(X,\xi)$ be a Borel $G$-space and $(Y,\kappa)$ a Borel $G_o$-space. \vspace{0.1cm}
\item $\alpha : (X,\xi) \ra (Y,\nu)$ a measurable $G_o$-map. \vspace{0.1cm}
\end{itemize}

\begin{proposition}
\label{prop_dense}
There exist
\begin{itemize}
\item a Borel $G$-space $(Z,\lambda)$, \vspace{0.1cm}
\item a measurable $G_o$-isomorphism $\beta : (Y,\kappa) \ra (Z,\lambda)$, \vspace{0.1cm}
\item a measurable $G$-map $\gamma : (X,\xi) \ra (Z,\lambda)$, \vspace{0.1cm}
\end{itemize} 
such that  \\
\vspace{0.2cm}
\centerline{\xymatrix{\ar[d]_{\alpha} (X,\xi) \ar[dr]^{\gamma}  & \\
(Y,\kappa)   \ar[r]_{\beta} &   (Z,\lambda) }}
\\
In other words, every $G_o$-factor of the Borel $G_o$-space $(X,\xi)$ is measurably $G_o$-isomorphic 
to a $G$-factor of the Borel $G$-space $(X,\xi)$, viewed as a $G_o$-space.
\end{proposition}

\begin{proof}
The proof consists of two main steps, which are both composed of references to the appendix of Zimmer's book \cite{Zi}. \\

\noindent \textsc{Step I: Construction of $(Z,\lambda)$ and $\beta$.} \\

Let $(W,\vartheta)$ be a Borel $G$-space.  Since $G$ is lcsc, $G$ acts continuously on the Banach space $L^1(W,\vartheta)$ and thus on 
$L^\infty(W,\vartheta)$ endowed with \emph{the weak*-topology} (the $G$-action is not necessarily continuous with respect to the norm-topology). We write
\[
B(W,\vartheta) = \big\{ f \in L^\infty(W,\vartheta) \, : \, f^2 = f \big\},
\]
which is clearly a $G$-invariant weak*-compact set. If $(W',\vartheta')$ is another Borel $G$-space, and $\delta : (W,\vartheta) \ra (W',\vartheta')$ is a measurable $G$-map,
then $\delta^* : B(W',\vartheta') \ra B(W,\vartheta)$ is a continuous, $G$-equivariant and \emph{injective} map. In particular, 
$\delta^*(B(W',\vartheta'))$ is a weak*-compact and $G$-invariant subset of $B(W,\vartheta)$. \\

Let us now specialize this to our setting. Since $G_o$ is dense in $G$ and $G$ acts continuously on $B(X,\xi)$ and 
since $\alpha^*(B(Y,\kappa)) \subset B(X,\xi)$ is $G_o$-invariant and weak*-compact, we conclude that $\alpha^*(B(Y,\kappa))$ is $G$-invariant. 
Since $\alpha^*$ is injective, we can thus endow $B(Y,\kappa)$ with a jointly continuous $G$-action so that $\alpha^*$ is 
$G$-equivariant, by setting 
\[
g.f := (\alpha^*)^{-1}(g.\alpha^*(f)), \quad \textrm{for $g \in G$ and $f \in B(Y,\kappa)$}.
\]
By \cite[Theorem B.10]{Zi}, there exist a Borel $G$-space $(Z,\lambda)$ and a $G$-equivariant Boolean map 
$\phi : B(Z,\lambda) \ra B(Y,\kappa)$. Since $\phi$ is also $G_o$-equivariant, \cite[Corollary B.7]{Zi} asserts that there exists a 
measurable $G_o$-isomorphism $\beta : (Y,\kappa) \ra (Z,\lambda)$, such that $\phi = \beta^*$.\\

\noindent \textsc{Step II: Construction of $\gamma$} \\

Note that $\psi : \alpha^* \circ \beta^* : B(Z,\eta) \ra B(X,\mu)$ is a $G$-equivariant Boolean $G$-map, so by \cite[Corollary B.6]{Zi},
there exists a measurable $G$-map $\gamma : (X,\mu) \ra (Z,\eta)$ such that $\gamma^* = \psi$, and thus $\gamma = \beta \circ \alpha$.
\end{proof}

We stress that the maps $\beta$ and $\gamma$ in Proposition \ref{prop_dense} are not measure-preserving in general. However, the following 
observation tells us that in the setting of \emph{ergodic} Borel $(G,\mu)$-spaces, for $\mu$ spread-out, these maps are automatically measure-preserving. 

\begin{lemma}
\label{lemma_dense}
Let $\mu$ be a spread-out probability measure on $G$ and let $(X,\xi)$ and $(X',\xi')$ be ergodic Borel $(G,\mu)$-spaces. 
Then every measurable $G$-map 
\[
\pi : (X,\xi) \ra (X',\xi')
\] 
is measure-preserving. 
\end{lemma}

\begin{proof}
By Proposition \ref{prop_cptmodel} we may without loss of generality assume that $X$ and $X'$ are compact $G$-spaces, and that $\pi$
is Borel. Since $\pi$ is a $G$-equivariant and $\xi$ is ergodic, $\pi_*\xi$ is a $\mu$-stationary and ergodic probability measure on $X'$ which is absolutely continuous with 
respect to $\xi'$, whence equal to $\xi'$ by \textsc{(iii)} in Lemma \ref{lemma_basicprops}.
\end{proof}

If we combine this lemma with Proposition \ref{prop_dense}, we arrive at the following corollary.

\begin{corollary}
\label{cor_dense}
Let $G$ be a lcsc group and $G_o$ a dense subgroup of $G$. Suppose that $\mu$ and $\mu_o$ are spread-out probability measures on $G$
and $G_o$ respectively, and let $(X,\xi)$ be an ergodic Borel $(G,\mu)$-space which is also a Borel $(G_o,\mu_o)$-space. Then every 
$G_o$-factor of $(X,\xi)$ is measurably $G_o$-isomorphic (in a measure-preserving way) to a $G$-factor of $(X,\xi)$.
\end{corollary}

\subsection{Transitive $G$-spaces}
\label{subsec:transitive}
Let $G$ be a lcsc group and $P$ a closed subgroup of $G$. By \cite[Section V.4]{Var}, the quotient space $G/P$ admits a unique $G$-invariant
measure class (but not necessarily a $G$-invariant \emph{measure}). By \cite[Section IV.2, Proposition 2.4b)]{Ma}, if $(X,\xi)$ is a Borel $G$-space and 
$\pi$ is meaurable $G$-equivariant map from $G/P$ to $X$, which maps the $G$-invariant measure class on $G/P$ to the measure class of $\xi$,
then $(X,\xi)$ is measurably $G$-isomorphic to a quotient space of the form $G/Q$, where $Q$ is a \emph{closed subgroup} of $G$ which 
contains $P$ (here, $G/Q$ is endowed with the unique $G$-invariant measure class). From this, and Corollary \ref{cor_dense} above, we can now
conclude the following result. 

\begin{corollary}
\label{cor_densesubgroupfactorhomo}
Let $G_o$ be a dense subgroup of $G$, $\mu_o$ a spread-out probability measure on $G_o$ and $\xi$ a $\mu_o$-stationary probability 
measure on $G/P$ which is absolutely continuous with respect to the unique $G$-invariant measure class. Then every $G_o$-factor of the
Borel $(G_o,\mu_o)$-space $(G/P,\xi)$ is 
$G_o$-isomorphic to $(G/Q,\xi')$, where $Q$ is a closed subgroup of $G$ which contains $P$, and $\xi'$ is $\mu_o$-stationary. 
In particular, if $P$ is a maximal proper closed subgroup of $G$, then the Borel $(G_o,\mu_o)$-space $(G/P,\xi)$ is prime.
\end{corollary}

\subsection{The Furstenberg entropy of Borel $(G,\mu)$-spaces}

Let $G$ be a lcsc group and let $\mu$ be a Borel probability measure on $G$ (not necessarily spread-out). Suppose that $(X,\xi)$ is a Borel space
endowed with a jointly measurable action of $G$, preserving the measure class of $\xi$. We assume  that $\xi$ is $\mu$-stationary.  
The \emph{Furstenberg entropy} $h_\mu(X,\xi)$ is defined by
\[
h_\mu(X,\xi) = \int_G \Big( \int_X -\log \frac{dg^{-1}\xi}{d\xi}(x) \, d\xi(x) \Big) \, d\mu(g),
\]
whenever the integral exists and is finite. Note that $h_{\delta_e}(X,\xi) = 0$. If $\mu$ is spread-out, we define the \emph{entropy $h(\mu)$}  as the Furstenberg entropy of the Poisson boundary of the measured group $(G,\mu)$. We define
\[
\cA_\xi = \big\{ \mu' \in \Prob(G) \, \mid \, \mu' * \xi = \xi, \enskip \textrm{$h_{\mu'}(X,\xi)$ exists and is finite} \big\}.
\]
We leave the proof of the following easy lemma to the reader.
\begin{lemma}
\label{lemma_propEntropy}
The set $\cA_\xi \subset \Prob(G)$ is convex and closed under convolution. Furthermore, the map $\mu' \mapsto h_{\mu'}(X,\xi)$ is affine on the $\cA_\xi$, 
and satisfies
\[
h_{\mu' * \mu''}(X,\xi) = h_{\mu'}(X,\xi) + h_{\mu''}(X,\xi), \quad \textrm{for all $\mu',\mu'' \in \cA_\xi$}.
\]
\end{lemma}

%%%%%%%%%%%%

\section{Proof of Theorem \ref{thm2}}

Throughout this section, let $\Gamma$ denote a countable discrete group, $H$ a totally disconnected lcsc group and 
$\rho : \Gamma \ra H$ a homomorphism with dense image. We shall fix a compact open subgroup $L$ of $H$ and 
set $\Lambda = \rho^{-1}(L) < \Gamma$. Then $(\Gamma,\Lambda)$ is a Hecke pair, and the map
\[
\overline{\rho} : \Gamma/\Lambda \ra H/L, \enskip \gamma \Lambda \mapsto \rho(\gamma)L
\]
is a bijection. We write $\alpha$ for the canonical quotient map from $\Gamma$ to $\Gamma/\Lambda$, and if $\varphi \in C_b(H)$, we
define $\varphi_L \in C_b(H/L)$ by
\[
\varphi_L(hL) = \int_L \varphi(hl) \, dm_L(l), \quad \textrm{for $hL \in H/L$},
\]
where $m_L$ denotes the Haar probability measure on $L$. We define the \emph{affine} map 
\[
\Prob(\Gamma ; \Lambda) \ra \Prob(H,L), \enskip \tau \mapsto \theta_\tau
\]
by
\begin{equation}
\label{defoftheta}
\theta_\tau(\varphi) = (\alpha_*\tau)(\varphi_L \circ \overline{\rho}), \quad \textrm{for $\varphi \in C_b(H)$}.
\end{equation}
By construction, $\theta_\tau$ is right-$L$-invariant. To prove that it is also left-$L$-invariant, and thus belongs to $\Prob(H,L)$, 
pick $l \in \rho(\Lambda)$ 
and $\lambda_l \in \Lambda$ such that $\rho(\lambda_l) = l$. Since $\tau$ is $\Lambda$-absorbing, $\tau_*\alpha$ is 
left-$\Lambda$-invariant, and thus
\[
\theta_\tau(\varphi(l \cdot)) = \alpha_*\tau((\varphi_L \circ \overline{\rho})(\lambda_l \cdot)) = 
\alpha_*\tau((\varphi_L \circ \overline{\rho})(\cdot)) = \theta_\tau(\varphi),
\]
for all $\varphi \in C_b(H)$, whence $\theta_\tau$ is left $\rho(\Lambda)$-invariant (and thus left $L$-invariant, since $\rho(\Lambda)$
is dense in $L$, and $H$-stabilizers of probability measures on $H$ are \emph{closed} subgroups). \\

It will be convenient to use the following notation: if $\tau$ is a $\Lambda$-absorbing probability measure on $\Gamma$,
set
\[
\overline{\tau} = \alpha_* \tau \qand \overline{\theta}_{\tau} = (\overline{\rho})_*\overline{\tau}, 
\]
so that $\overline{\theta}_{\tau}(\varphi_L) = \theta_{\tau}(\varphi)$ for all $\varphi \in C_b(H)$. In particular, 
\begin{equation}
\label{useful}
\overline{\tau}(\varphi_L \circ \overline{\rho}) = \theta_{\tau}(\varphi), \quad \textrm{for all $\varphi \in C(H) \cap \cL^1(H,\theta_\tau)$}.
\end{equation}
Since $\overline{\rho}$ is a bijection and $\tau \mapsto \overline{\tau}$ is surjective, we conclude that $\tau \mapsto \theta_\tau$
is surjective. 
 
\subsection*{Proof that the map is a monoid homomorphism}
 
We recall from \eqref{leftinv} that 
\[
\alpha_*(\tau_1 * \tau_2)(\gamma \Lambda) = \sum_{\eta \Lambda} \overline{\tau}_1(\eta \Lambda) \, \overline{\tau}_2(\eta^{-1} \gamma \Lambda), \quad \textrm{for $\gamma \Lambda \in \Gamma/\Lambda$}.
\]
Hence, for all $\varphi \in C_b(H)$, 
\begin{eqnarray*}
\theta_{\tau_1 * \tau_2}(\varphi) 
&=&
\alpha_*(\tau_1 * \tau_2)(\varphi_L \circ \overline{\rho}) = \sum_{\gamma \Lambda} \sum_{\eta \Lambda} 
\varphi_L(\overline{\rho}(\gamma \Lambda)) \, \overline{\tau}_1(\eta \Lambda) \, \overline{\tau}_2(\eta^{-1} \gamma \Lambda) \\[0.2cm]
&=&
\sum_{\eta \Lambda} \Big( \sum_{\gamma \Lambda} \varphi_L(\rho(\eta) \overline{\rho}(\gamma \Lambda)) \, \overline{\tau}_2(\gamma \Lambda)
\Big) \, \overline{\tau}_1(\eta \Lambda) \\[0.2cm]
&=&
\sum_{\eta \Lambda}  \overline{\theta}_{\tau_2}(\varphi_L(\rho(\eta) \cdot)) \, \overline{\tau}_1(\eta \Lambda)
= 
\sum_{\eta \Lambda}  \overline{\theta}_{\tau_2}((\varphi(\rho(\eta) \cdot))_L) \, \overline{\tau}_1(\eta \Lambda) \\[0.2cm]
&=&
\int_{H/L} \overline{\theta}_{\tau_2}((\varphi(h \cdot))_L) \, d\overline{\theta}_{\tau_1}(hL) = (\theta_{\tau_1} * \theta_{\tau_2})(\varphi),
\end{eqnarray*}
where we in the fourth identity have used \eqref{useful} and in the fifth identity the fact that the assignment $\varphi \mapsto \varphi_L$ is left-$H$-equivariant, i.e. $\varphi_L(h\cdot) = (\varphi(h\cdot))_L$ for all $h \in H$. Since $\tau_1$ and $\tau_2$ are arbitrary in $\Prob(\Gamma ; \Lambda)$, we conclude that 
\[
\theta_{\tau_1 * \tau_2} = \theta_{\tau_1} * \theta_{\tau_2}, \quad \textrm{for all $\tau_1, \tau_2 \in \Prob(\Gamma ; \Lambda)$}.
\]
Finally, $\theta_{\delta_e} = m_L$ and thus the map $\tau \mapsto \theta_\tau$ is a monoid homomorphism.
 
\subsection*{Proof of Property \textsc{(P1)}} 
 
Suppose that $\varphi \in C(H) \cap \cL^1(H,\theta_\tau)$ is right $L$-invariant. Then,
\[
\varphi_L \circ \overline{\rho} \circ \alpha = \varphi \circ \rho,
\]
whence $\theta_\tau(\varphi) = \rho_*\tau(\varphi)$ by \eqref{useful}. 
 
\subsection*{Proof of Property \textsc{(P2)}}

Let us fix $\tau \in \Prob(\Gamma ; \Lambda)$. Since $\theta_\tau$ is bi-$L$-invariant, its support must be bi-$L$-invariant, and since 
$\rho(\Gamma)$ is dense in $H$, it suffices to
prove the inclusions
\[
\rho(\supp \tau) \subset
\supp \theta_\tau \cap \rho(\Gamma) \subset \rho(\supp(\tau))L,
\]
or equivalently, that
\[
\rho(\gamma) \in \supp \theta_\tau \iff \rho(\gamma)L \subset \rho(\supp \tau)L.
\]
To show this, pick $\gamma \in \Gamma$ and an open identity neighborhood in $H$, which we may assume is contained in $L$. 
If we unwrap the definition of the measure $\theta_\tau$ (see \eqref{defoftheta}), 
we get
\begin{eqnarray*}
\theta_\tau(\rho(\gamma)U) 
&=&
\int_{H/L} \Big( \int_L \chi_{\rho(\gamma)U}(hl) \, dm_L(l) \Big) \, d\overline{\theta}_\tau(hL) \\[0.2cm]
&=&
\sum_{\eta \Lambda} \Big( \int_L \chi_{\rho(\gamma)U}(\rho(\eta) l) \, dm_L(l) \Big) \, \overline{\tau}(\eta \Lambda) \\[0.2cm]
&=&
\sum_{\eta \Lambda} \sum_{\lambda \in \Lambda} \Big( \int_L \chi_{\rho(\gamma)U}(\rho(\eta) l) \, dm_L(l) \Big) \, \tau(\eta \lambda) \, 
\end{eqnarray*}
and thus $\theta_\tau(\rho(\gamma)U) > 0$ if and only if there exist $\eta \in \Gamma$ and $\lambda \in \Lambda$ such that
\[
m_L(\rho(\eta^{-1}\gamma)U \cap L) > 0 \qand \tau(\eta \lambda) > 0.
\]
Since $U$ is open, the first condition is equivalent to saying that $\rho(\eta^{-1}\gamma)U \cap L \neq \emptyset$. Since the second condition is left-$\Lambda$-invariant, and $U$ is an arbitrary clopen identity neighborhood in $L \subset H$, we conclude that
\[
\rho(\gamma) \in \supp \theta_\tau \iff \rho(\gamma)L \subset \rho(\supp \tau)L,
\]
which finishes the proof.

\subsection*{Proof of Property \textsc{(P3)}}

Let $X$ be a Borel $H$-space, $\xi$ a $L$-invariant Borel measure on $X$ and $\tau$ a $\Lambda$-absorbing probability measure on $\Gamma$.
For every bounded Borel function $f$ on $X$, the function 
\[
\varphi_f(h) = h \xi(f), \quad \textrm{$h \in H$}
\] 
is right $L$-invariant, and thus continuous since $H$ is open, whence by \textsc{(P1)} above, 
\[
(\theta_\tau * \xi)(f) = \int_H \varphi_f(h) \, d\theta_\tau(h) = \sum_{\gamma \in \Gamma} \tau(\gamma) \, \rho(\gamma)\xi(f) = (\tau * \xi)(f).
\]
Since $f$ is arbitrary, we conclude that $\theta_\tau * \xi = \tau * \xi$.

%We pick a $\Lambda$-absorbing probability measure
%$\tau$ on $\Gamma$ and note that since $\rho(\Lambda) \subset L$, 
%\begin{eqnarray*}
%\tau * \xi 
%&=& 
%\sum_{\gamma} \tau(\gamma) \, \rho(\gamma)_*\xi = \sum_{\gamma \Lambda} \overline{\tau}(\gamma) \, \rho(\gamma)_*\xi \\[0.2cm]
%&=&
%\int_{H/L} h_*\xi \, d\overline{\theta}_\tau(hL) = \int_H h_*\xi \, d\theta_\tau(h) =
%\theta_\tau * \xi,
%\end{eqnarray*}
%which finishes the proof.

\subsection*{Proof of Property \textsc{(P4)}}

Let $\tau$ be a $\Lambda$-absorbing probability measure on $\Gamma$ and fix a Borel $(H,\theta_\tau)$-space $(X,\xi)$. Since 
$\xi$ is $L$-invariant, so is $\xi^{\otimes 4}$ under the diagonal $H$-action on $X^4$. Using that $\tau \mapsto \theta_\tau$
is a monoid homomorphism and property (P3) above, we see that
\[
\tau^{*n} * \xi^{\otimes 4} = \theta_{\tau^{*n}} * \xi^{\otimes 4} = \theta_{\tau}^{*n} * \xi^{\otimes 4}, \quad \textrm{for all $n \geq 1$},
\]
and thus $\xi$ is $\tau$-proximal if and only if it is $\theta_\tau$-proximal by \textsc{(i)} in Proposition \ref{prop_proxcrit}.

\begin{remark}
\label{remark_prox}
Let us highlight a hidden subtlety in the proof of \textsc{(P4)}. In view of our Lemma \ref{lemma_proxtransfer} above, the key point 
is really to
show that if $(M,\eta)$ is a compact model of a $\theta_\tau$-proximal Borel $(H,\theta_\tau)$-space, then $\eta_{\Delta_2}$ is the
only \emph{$\tau$-stationary} $2$-coupling of $(M,\eta)$. However, a compact $H$-space could in general admit $\tau$-stationary 
probability measures, which are not $\theta_\tau$-stationary. To circumvent this issue, we found it necessary to establish Proposition 
\ref{prop_proxcrit}. We stress that we are not aware of an argument which would prove \textsc{(P4)}
for an \emph{arbitrary} Furstenberg discretization of $\theta_\tau$. 
\end{remark}

\section{Proof of Corollary \ref{cor1_Thm2}}
\label{sec:proofCor1Thm2}

Let $(\Gamma,\Lambda)$ be a Hecke pair and $(H,L,\rho)$ a Hecke completion of $(\Gamma,\Lambda)$. Fix a spread-out and 
$\Lambda$-absorbing probability measure $\tau$ on $\Gamma$, and write $(H,\theta_\tau)$ for the Hecke completion of 
$(\Gamma,\tau)$ with respect to $\rho$. Let $(B_\tau,\nu_\tau)$ and $(B_{\theta_\tau},\nu_{\tau})$ denote the Poisson boundaries 
of the measured groups $(\Gamma,\tau)$ and $(H,\theta_\tau)$
respectively.

\subsection*{Proof of \textsc{(I)}}
 
Since $\nu_\tau$ is $\theta_\tau$-proximal, it is also $\tau$-proximal by \textsc{(P4)} in Theorem \ref{thm2}, and thus $(B_{\theta_\tau},\nu_{\tau})$
is a $\tau$-boundary.
 
\subsection*{Proof of \textsc{(II)}}

Let $M$ be a compact $\Gamma$-space and $\eta$ a $\tau$-stationary probability measure on $M$. Suppose that $\eta$ is $\Lambda$-invariant.

\begin{lemma}
\label{lemma_Texists}
There is a $\Gamma$-equivariant, linear and unital bounded map 
\[
T_\eta : C(M) \ra L^\infty(B_{\theta_\tau},\nu_{\theta_\tau}) 
\]
with the property that
\[
\gamma \eta(f) = \rho(\gamma)\nu_{\theta_\tau}(T_\eta f), \quad \textrm{for all $\gamma \in \Gamma$ and $f \in C(M)$}.
\]
\end{lemma}

\begin{proof}
Since $\eta$ is $\Lambda$-invariant, for every $f \in C(M)$, the Poisson transform $P_\eta f$ is a bounded, $\tau$-harmonic and right $\Lambda$-invariant
function on $\Gamma$. Since 
\[
\overline{\rho} : \Gamma/\Lambda \ra H/L, \enskip \gamma \Lambda \mapsto \rho(\gamma)L,
\]
is a bijection, $\varphi_f(h) := P_\eta f(\overline{\rho}^{-1}(hL))$ is a well-defined bounded and right $L$-invariant (hence continuous) 
function on $H$, which satisfies $\varphi_f \circ \rho = P_\eta f$. By Property \textsc{(P1)} in Theorem \ref{thm2}, applied to left-translates of $\varphi_f$, we conclude that $\varphi_f$ is $\theta_\tau$-harmonic. By Theorem \ref{thm_Poisson}, the Poisson transform $P_{\nu_{\theta_\tau}}$ sets up an isometric bijection between $L^\infty(B_{\theta_\tau},\nu_{\theta_\tau})$ and $\cH^\infty(H,\theta_\tau)$, so we can define $T_\eta : C(M) \ra L^\infty(B_{\theta_\tau},\nu_{\theta_\tau})$ by
\[
T_\eta f = P_{\nu_{\theta_\tau}}^{-1}\varphi_f, \quad \textrm{for $f \in C(M)$},
\]
which is clearly $\Gamma$-equivariant and bounded, and satisfies $\gamma \eta(f) = \rho(\gamma)\nu_{\theta_\tau}(T_\eta f)$ for all $\gamma \in \Gamma$ and $f \in C(M)$.
\end{proof}

We recall the following notation from Subsection \ref{subsec:canonicalquasi}: if $f \in C(M)$ and $\kappa \in \Prob(M)$, 
we define $\widehat{f}(\kappa) = \kappa(f)$. If we endow $\Prob(M)$ with the weak*-topology, $\widehat{f}$ is a continuous 
function on $\Prob(M)$.

\begin{corollary}
\label{cor_Texists}
There is a measure-preserving $\Gamma$-map $\pi_\eta : (B_{\theta_\tau},\nu_{\theta}) \ra (\Prob(M),(\pi_\eta)_*\nu_\theta)$ such that
\[
\widehat{f} \circ \pi_\eta = T_\eta f, \quad \textrm{for all $f \in C(M)$},
\]
in the $L^\infty$-sense. 
\end{corollary}

\begin{proof}
Since $M$ is compact and metrizable, $C(M)$ (equipped with the uniform norm) is separable Banach space. Fix a countable dense
subset $(f_i)$ in $C(M)$. Since $\Gamma$ is countable, we may without loss of generality assume that this set is closed under composition
with elements from $\Gamma$. Let us choose, for every $i$, a 
Borel function $\widetilde{T_\eta f_i}$ on $B_{\theta_\tau}$ which represents the function class $T_\eta f_i$ in $L^\infty(B_{\theta_\tau},\nu_{\theta_\tau})$.
Since $\Gamma$ is countable, we can find a $\Gamma$-invariant $\nu_{\theta_\tau}$-conull subset $B' \subset B_{\theta_\tau}$ such that for every $i$
and $\gamma \in \Gamma$,
\begin{equation}
\label{preleq}
\widetilde{T_\eta(f_i \circ \gamma)}(b) = \widetilde{T_\eta f_i}(\gamma b), \quad \textrm{for all $b \in B'$}.
\end{equation}
One readily checks that the functionals on the $\bQ$-linear span of $(f_i)$, defined by
\[
\pi_\eta(b)(f_i) := \widetilde{T_\eta f_i}(b), \quad \textrm{for $b \in B'$},
\]
for every $i$, are bounded, positive and unital, and thus extend by a straightforward approximation argument, to probability measures on $M$ (via Riesz representation theorem). By dominated convergence and \eqref{preleq}, $\pi_\eta$ is a $\Gamma$-equivariant map from $B'$ to $\Prob(M)$, 
which satisfies $\widehat{f} \circ \pi_\eta = T_\eta f$ for all $f \in C(M)$ (in the $L^\infty$-sense) by construction.
\end{proof}

It remains to prove that $(\pi_\eta)_*\nu_{\theta_\tau} = \eta^*$, where $\eta^*$ denote the canonical quasi-factor associated with $(M,\eta)$, as
introduced in Subsection \ref{subsec:canonicalquasi}. By the discussion in this subsection, this amounts to showing
\[
\eta^*(\widehat{f}_1 \cdots \widehat{f}_k)  = (\pi_\eta)_*\nu_{\theta_\tau}(\widehat{f}_1 \cdots \widehat{f}_k),
\]
for all $f_1,\ldots,f_k \in C(M)$ and $k \geq 1$. By Corollary \ref{cor_Texists}, we have
\[
(\pi_\eta)_*\nu_{\theta_\tau}(\widehat{f}_1 \cdots \widehat{f}_k) 
=
\nu_{\theta_\tau}(\widehat{f_1} \circ \pi_{\eta} \cdots \widehat{f_k} \circ \pi_{\eta}) \\[0.2cm] 
=
\nu_{\theta_\tau}(T_\eta f_1 \cdots T_\eta f_k),
\]
for all $f_1,\ldots,f_k \in C(M)$ and $k \geq 1$, so it suffices to show that 
\[
\eta^*(\widehat{f}_1 \cdots \widehat{f}_k) = \nu_{\theta_\tau}(T_\eta f_1 \cdots T_\eta f_k).
\]
Let us from now on fix $k$ and a $k$-tuple $f_1,\ldots,f_k$ in $C(M)$. By Lemma \ref{lemma_canonicalquasifactor},
\[
\eta^*(\widehat{f}_1 \cdots \widehat{f}_k) 
=
\lim_{n \ra \infty} (\tau^{*n} * \eta^{\otimes k})(f_1 \otimes \cdots \otimes f_k).
\]
Using Lemma \ref{lemma_Texists}, we see that
\begin{eqnarray*}
(\tau^{*n} * \eta^{\otimes k})(f_1 \otimes \cdots \otimes f_k) 
&=&
\sum_{\gamma \in \Gamma} \gamma \eta(f_1) \cdots \gamma \eta(f_k) \, \tau^{*n}(\gamma) \\[0.2cm]
&=&
\sum_{\gamma \in \Gamma} \rho(\gamma) \nu_{\theta_\tau}(T_\eta f_1) \cdots \rho(\gamma) \nu_{\theta_{\tau}}(T_\eta f_k) \, \tau^{*n}(\gamma),
\end{eqnarray*}
for all $n$. Since $\theta_\tau$ is $L$-invariant, $h \mapsto h \nu_{\theta_\tau}(f_j)$ is right $L$-invariant for every $j = 1,\ldots,k$, 
and thus by Property \textsc{(P1)} in Theorem \ref{thm2}, we have
\begin{eqnarray*}
\sum_{\gamma \in \Gamma} \rho(\gamma) \nu_{\theta_\tau}(T_\eta f_1) \cdots \rho(\gamma) \nu_{\theta_{\tau}}(T_\eta f_k) \, \tau^{*n}(\gamma)
&=&
\int_H h \nu_{\theta_\tau}(T_\eta f_1) \cdots h \nu_{\theta_{\tau}}(T_\eta f_k) \, d\theta_{\tau}^{*n}(h) \\
&=&
(\theta_\tau^{*n} * \nu_{\theta_\tau}^{\otimes k})(T_\eta f_1 \otimes \cdots \otimes T_\eta f_k),
\end{eqnarray*}
for all $n$. Since $(B_{\theta_\tau},\nu_\theta)$ is $\theta_\tau$-proximal, Proposition \ref{prop_proximalktuples} tells us that
\[
\lim_n (\theta_\tau^{*n} * \nu_{\theta_\tau}^{\otimes k})(T_\eta f_1 \otimes \cdots \otimes T_\eta f_k) = \nu_{\theta_\tau}(T_\eta f_1 \cdots T_\eta f_k),
\]
and thus
\[
\lim_n (\tau^{*n} * \eta^{\otimes k})(f_1 \otimes \cdots \otimes f_k) = \nu_{\theta_\tau}(T_\eta f_1 \cdots T_\eta f_k).
\]
From the discussion above, this proves that $(\pi_\eta)_*\nu_{\theta_\tau} = \eta^*$. \\

Finally, if $(M,\eta)$ is $\tau$-proximal, then Proposition \ref{prop_proximalityandcanonicalquasi} \textsc{(II)} shows that $(\Prob(M),\eta^*)$ is 
$\Gamma$-isomorphic to $(M,\eta)$, whence the argument above shows that $(M,\eta)$ is a $\theta_\tau$-boundary. 

\subsection*{Proof of \textsc{(III)}}

The $\Gamma$-action on $(B_\tau,\nu_\tau)$ is amenable by \cite[Theorem 5.2]{Zi0}, and thus the $\Lambda$-action is amenable as well by
\cite[Theorem 4.3.5]{Zi}. If $(B_\tau,\nu_\tau)$ is measurably  $\Gamma$-isomorphic to $(B_{\theta_\tau},\nu_{\theta_\tau})$, then, since 
$\nu_{\theta_\tau}$ is $L$-invariant and $\Lambda = \rho^{-1}(L)$, we see that $\nu_\tau$ is $\Lambda$-invariant. 
By \cite[Proposition 4.3.3]{Zi}, this forces $\Lambda$ to be an amenable group.

\subsection*{Proof of \textsc{(IV)}}

Suppose that $\Lambda$ is an amenable group and that $(B_\tau,\nu_\tau)$ admits a compact model $(M,\eta)$ which is uniquely 
$\tau$-stationary. Since $\tau$ is $\Lambda$-absorbing and $\Lambda$ is amenable, \textsc{(ii)} in Theorem \ref{thm1} applied to the 
weak*-compact set $C = \Prob(M)$, asserts that there is at least one $\tau$-stationary and $\Lambda$-invariant probability measure 
on $M$. By unique $\tau$-stationarity, we conclude that $\eta$ is $\Lambda$-invariant. Since $\eta$ is $\tau$-proximal, 
\textsc{(II)} above implies that $(M,\eta)$ is a $\theta_\tau$-boundary. Furthermore, by \textsc{(I)} 
above, $(B_{\theta_\tau},\nu_{\theta_\tau})$ is a $\tau$-boundary. We conclude that there are measure-preserving $\Gamma$-maps
\[
(B_\tau,\nu_\tau) \overset{\pi_1}\longrightarrow (B_{\theta_\tau},\nu_{\theta_\tau}) \overset{\pi_2}\longrightarrow (B_\tau,\nu_\tau).
\]
In particular, $\pi := \pi_2 \circ \pi_1 : (B_\tau,\nu_\tau) \ra (B_\tau,\nu_\tau)$ is a measure-preserving $\Gamma$-map, whence, 
by \cite[Proposition 3.2]{FuGl},  almost everywhere equal to the identity map, and thus $(B_\tau,\nu_\tau)$ and $(B_{\theta_\tau},\nu_{\theta_\tau})$ 
are measurably $\Gamma$-isomorphic. 

\section{Proof of Corollary \ref{cor2_thm2}}
\label{sec:cor2thm2}

Let $p \geq 2$ be a prime number. By \cite[Corollary 1.4]{BrGe}, there is a free group $\Gamma$ of finite rank $r = r(p)$ and an \emph{injective} homomorphism $\rho : \Gamma \ra \SL_2(\bQ_p)$ with dense image (one can also appeal to Ihara's Theorem \cite{IY} which in particular implies that every torsion-free subgroup of $\SL_2(\bQ_p)$ is free). Set 
\[
H = \SL_2(\bQ_p) \qand L = \SL_2(\bZ_p) \qand \Lambda = \rho^{-1}(L).
\]
Then $(\Gamma,\Lambda)$ is a Hecke pair and $(H,L,\rho)$ is a completion triple of $(\Gamma,\Lambda)$. Let $P$ denote the closed solvable (hence amenable) subgroup of $H$ consisting of upper triangular matrices. Since $H = LP$, there is a unique $L$-invariant Borel probability measure $\nu_o$
on $H/P$. Let $\tau$  be a 
spread-out and finitely supported $\Lambda$-absorbing probability measure on $\Gamma$. Such a measure exists by \textsc{(III)} in Theorem \ref{thm1}. By \textsc{(P2)} in Theorem \ref{thm2}, $\theta_\tau$ is spread-out and compactly supported, and by \cite[Theorem 15.5]{GuJiTa}, the quotient space 
$(H/P,\nu_o)$ is the Poisson boundary of $(H,\theta_\tau)$. By \textsc{(I)} in Corollary \ref{thm1}, $(H/P,\nu_o)$ is a $\tau$-boundary.  We note that
\[
\Stab_\Gamma(hP) = \rho^{-1}(\Gamma \cap hPh^{-1}), \quad \textrm{for all $hP \in H/P$}.
\]
In particular, $\Stab_\Gamma(hP)$ is an amenable subgroup of $\Gamma$ for every $hP \in H/P$, since $\rho$ is injective and $P$ is solvable. Since
a free group of finite rank only has countably many amenable subgroups (the trivial group and cyclic subgroups), the map $hP \mapsto \Stab_\Gamma(hP)$
has a countable image, whence must be essentially trivial (see e.g. \cite[Proposition 1.3]{BoHaTa}), and thus the $\Gamma$-action on $(H/P,\nu_o)$ is essentially free.

%XXXXXXXXX \\\
%
%We shall now show that $\Stab_\Gamma(hP)$ is trivial $\nu_o$-almost everywhere. This can be done in different ways; to make referencing 
%easier, we use an approach by the second author and Kalantar \cite{HaKa}. Let $\partial \Gamma$ denote the space of ends of $\Gamma$. As is well-known (see e.g. \cite{Kai}), there exists a unique 
%$\tau$-stationary probability measure $\nu_\tau$ on $\partial \Gamma$. Furthermore, $(\partial \Gamma,\nu_\tau)$ is the Poisson boundary
%of the measured group $(\Gamma,\tau)$, and the $\Gamma$-action on $\partial \Gamma$ is essentially free (with respect to $\nu_\tau$). 
%By \cite[Theorem 4.12]{HaKa}, this implies that $\tau$ is a C*-simple measure, whence by \cite[Theorem 5.3]{HaKa}, every $\tau$-boundary with amenable stabilizers must be essentially free. We conclude that $(H/P,\nu_o)$ is an essentially free $\tau$-boundary.

It remains to prove that $(H/P,\nu_o)$ is a \emph{prime} $(\Gamma,\tau)$-space. By Corollary \ref{cor_densesubgroupfactorhomo} this amounts to 
showing that $P$ is a maximal proper closed subgroup of $H$. The \emph{Bruhat decomposition} of $H$ with respect to $P$ tells us that
\[
H = P \, \sqcup \, PwP, \quad \textrm{where} \quad
w = \left(
\begin{matrix} 0 & -1 \\ 1 & 0 \end{matrix}
\right).
\]
This decomposition demonstrates two things. Firstly, the group generated by $P$ and $w$ equals $H$. Secondly, any subgroup which strictly contains
$P$ must contain the element $w$, whence also the group generated by $P$ and $w$. In particular, $P$ is a maximal (not necessarily closed) subgroup of $H$, and we are done.

\section{Proof of Corollary \ref{cor3_thm2}}
\label{sec:cor3thm2}
Let $(\Gamma,\Lambda)$ be a Hecke pair. We shall assume that $\Lambda$ is non-amenable and not co-amenable in $\Gamma$. The latter
assumption guarantees by \cite[Proposition 3.4]{AD} that there exists a completion triple $(H,L,\rho)$ of $(\Gamma,\Lambda)$ with
$H$ \emph{non-amenable}.  

Let us fix a spread-out $\Lambda$-absorbing probability measure on $\Gamma$. Since $\Lambda$ is non-amenable,
so is $\Gamma$, and thus the Poisson boundary $(B_\tau,\nu_\tau)$ of $(\Gamma,\tau)$ is non-trivial. Let $(H,\theta_\tau)$ denote the 
Hecke completion of $(\Gamma,\tau)$ with respect to $\rho$. Since $H$ is non-amenable, the Poisson boundary 
$(B_{\theta_\tau},\nu_{\theta_\tau})$ of $(H,\theta_\tau)$ is non-trivial. By \textsc{(I)} in Corollary \ref{cor1_Thm2}, 
$(B_{\theta_\tau},\nu_{\theta_\tau})$ is thus a non-trivial $\tau$-boundary. Since $\Lambda$ is non-amenable, Corollary \ref{cor1_Thm2} \textsc{(III)} 
tells us $(B_{\tau},\nu_{\tau})$ and $(B_{\theta_\tau},\nu_{\theta_\tau})$ cannot be measurably $\Gamma$-isomorphic. We conclude that 
$(\Gamma,\tau)$ is not a prime measured group.

\section{Proof of Theorem \ref{thm3}}

Let $(K,|\cdot|)$ be a non-Archimedean local field with a prime residue field, and fix a uniformizer $\varpi$ of $K$. We retain the notation from 
Subsection \ref{subsec:grpsaffine}. In particular,
\[
\Gamma = \Xi \rtimes \langle \varpi \rangle \qand \Lambda = \Xi_o \rtimes \{1\},
\]
and
\[
H = K \rtimes  \langle \varpi \rangle \qand L = \cO \rtimes \{1\},
\]
where $\Xi_o$ and $\Xi$ are the additive subgroups of $K$ generated by all powers of $\varpi$ and all positive powers of $\varpi$ respectively.
We shall think of $\Gamma$ as a dense subgroup of $H$, and we recall that $H$ acts transitively on $K$ via
\[
(x,\varpi^n)y = x + \varpi^n y, \quad \textrm{for $(x,\varpi^n) \in H$ and $y \in K$},
\]
so that $K \cong H/P$, with $P = \{0\} \rtimes \langle \varpi \rangle$. 

\begin{lemma}
\label{lemma_Pismaximal}
$P$ is a maximal closed subgroup of $H$.
\end{lemma}

\begin{proof}
Suppose that $Q$ is a proper closed subgroup which strictly contains $P$. Then there exists a non-zero $x \in K$ such that 
$P(x,1)P \subset Q$, and thus 
\[
x\Xi \rtimes \langle \varpi \rangle \subset Q,
\]
since $\Xi$ is equal the additive group generated by all powers of $\varpi$. By assumption, $\Xi$ is dense in $K$, whence the left-hand side
is dense in $H$. Since $Q$ is closed, we see that $Q = H$.

\end{proof}

Throughout the rest of this section, we fix a finitely supported, spread-out and $\Lambda$-absorbing probability measure $\tau$ on $\Gamma$ which
satisfies 
\begin{equation}
\label{negdrGamma}
\sum_{n \in \bZ} n \, (\pr_{\bZ})_*\tau(n) < 0,
\end{equation}
where $\pr_{\bZ}$ is defined in \eqref{def_prZ}. We write $(H,\theta_\tau)$ for the 
Hecke completion of the Hecke measured group $(\Gamma,\Lambda)$, and note that since $\pr_{\bZ}$ is a homomorphism, 
\begin{equation}
\label{negdrH}
\sum_{n \in \bZ} n \, (\pr_{\bZ})_*\theta_\tau(n) < 0,
\end{equation}
by \textsc{(P1)} in Theorem \ref{thm2}.

\subsection{Proofs of (\textsc{I}) and (\textsc{II}) and (\textsc{III})}

Since both $\Gamma$ and $H$ are non-exceptional subgroups of the affine group of $K$ (in the sense of \cite{CKW}), the supports of 
$\tau$ and $\theta_\tau$ generate (as semi-groups) $\Gamma$ and $H$ respectively and \eqref{negdrGamma} and \eqref{negdrH} hold,  
\cite[Theorem 2]{CKW} tells us that \vspace{0.1cm}
\begin{itemize}
\item[\textsc{(F1)}] there is a unique $\tau$-stationary probability measure $\nu$ on $K$. \vspace{0.1cm}
\item[\textsc{(F2)}] there is a unique $\theta_\tau$-stationary probability measure $\nu'$ on $K$. \vspace{0.1cm}
\item[\textsc{(F3)}] $\nu$ is $\tau$-proximal and $\nu'$ is $\theta_\tau$-proximal.
\end{itemize}
\vspace{0.1cm}
By \textsc{(P3)} in Theorem \ref{thm2}, $\nu'$ is also $\tau$-stationary, whence $\nu = \nu'$ by uniqueness. Furthermore, since $\theta_\tau$
is spread-out, $\nu$ is $H$-quasi-invariant by \textsc{(i)} in Lemma \ref{lemma_basicprops}. Since $H$ acts transitively on $K$ and the Haar 
measure $m_K$ on $K$ is $H$-quasi-invariant, we conclude (from the uniqueness of $H$-invariant measure classes) that $\nu$ is absolutely
continuous with respect to $m_K$. We can thus write $d \nu = u dm_K$ for some non-negative $u \in L^1(K,m_K)$. Since $\theta_\tau$ is 
bi-$L$-invariant, and both $\nu$ and $m_K$ are $L$-invariant, $u$ is $L$-invariant and thus (almost everywhere equal to) a continuous function 
on $K$, which we still denote by $u$. We claim that $u$ is strictly positive. If it were not, then since it is $L$-invariant, the zero set of $u$ must be 
a non-empty $\cO$-invariant set, and thus $\nu(\cO + t) = 0$ for \emph{some} $t \in K$. However, since $\nu$ is $H$-quasi-invariant, this readily implies that $\nu(\cO+t) = 0$ for \emph{all} $t \in K$. However, this leads to a contradiction. Indeed, since $\cO$ is an open (additive) subgroup $K$, the quotient 
space $K/\cO$ must be countable. In particular, we can find a countable set $T \subset K$ such that $K = \cup_{t \in T} (\cO + t)$. Hence,
\[
1 = \nu(K) = \nu\big(\bigcup_{t \in T} \big(\cO + t\big)\big) \leq \sum_{t \in T} \nu(\cO + t) = 0,
\]
and thus we conclude that $u$ is strictly positive. Finally, by Lemma \ref{lemma_Pismaximal}, $P$ is a maximal closed subgroup of $H$,  and 
since $K \cong H/P$, Corollary \ref{cor_densesubgroupfactorhomo} tells us that the Borel $(\Gamma,\tau)$-space $(K,\nu)$ is prime.

\subsection{Proof of (\textsc{IV})}

We begin with a  few general remarks. Let $G$ be a lcsc group and suppose that $(X,\cB_X)$ is a measurable space endowed with a jointly
measurable action of $G$. Let $\xi$ be a $G$-quasi-invariant $\sigma$-finite measure on $\cB_X$.  Fix $p \geq 1$ be a real number. The 
\emph{regular $L^p$-representation} $(L^p(X,\xi),\sigma_{p,\xi})$ is given by
\[
\sigma_{p,\xi}(g) f = \Big(\frac{dg\xi}{d\xi}\Big)^{1/p} f \circ g^{-1}, \quad \textrm{for $g \in G$ and $f \in L^p(X,\xi)$}.
\]
We say that $(L^p(X,\xi),\sigma_{p,\xi})$ is \emph{$L^p$-irreducible} if for every non-zero $f \in L^p(X,\xi)$, the linear span of the set 
$\sigma_{\xi,p}(G)f$ is norm-dense in $L^p(X,\xi)$. Equivalently, $(L^p(X,\xi),\sigma_{p,\xi})$
is \emph{not} $L^p$-irreducible if there exist non-zero $f_1 \in L^p(X,\xi)$ and $f_2 \in L^{p'}(X,\xi)$, where $\frac{1}{p} + \frac{1}{p'} = 1$,
such that
\begin{equation}
\label{Lpreducible}
\int_X \Big(\frac{dg\xi}{d\xi}\Big)^{1/p} f_1 \circ g^{-1} \, \overline{f_2} \, d\xi = 0, \quad \textrm{for all $g \in G$}.
\end{equation}
We note that \eqref{Lpreducible} only depends on the measure class of $\xi$: suppose that $\xi'$ is another $G$-quasi-invariant 
$\sigma$-finite measure on $X$, which is equivalent to $\xi$, and write $d\xi' = u \, d\xi$, for some ($\xi$-almost everywhere) strictly positive function $u$. Then,
\[
F_1 := u^{-1/p} f_1 \in L^p(X,\xi') \qand F_2 := u^{-1/p'} f_2 \in L^{p'}(X,\xi'),
\]
and for every $g \in G$,
\[
\frac{dg\xi'}{d\xi'} = \frac{u \circ g^{-1}}{u} \, \frac{dg\xi}{d\xi}, \quad \textrm{$\xi$-almost everywhere},
\]
which readily implies that, for every $g \in G$,
\[
\int_X \Big(\frac{dg\xi'}{d\xi'}\Big)^{1/p} F_1 \circ g^{-1} \, \overline{F_2} \, d\xi'
\]
equals
\[
\int_X \Big(\frac{u \circ g^{-1}}{u}\Big)^{1/p} \, \Big(\frac{dg\xi}{d\xi}\Big)^{1/p} \, (u \circ g^{-1})^{-1/p} f_1 \circ g^{-1} \, u^{-1/p'} \overline{f_2} \,u d\xi
\]
which simplifies to
\[
\int_X \Big(\frac{dg\xi}{d\xi}\Big)^{1/p} f_1 \circ g^{-1} \, \overline{f_2} \, d\xi,
\]
since $\frac{1}{p} + \frac{1}{p'} = 1$.
In particular, 
\[
\int_X \Big(\frac{dg\xi'}{d\xi'}\Big)^{1/p} F_1 \circ g^{-1} \, \overline{F_2} \, d\xi' = \int_X \Big(\frac{dg\xi}{d\xi}\Big)^{1/p} f_1 \circ g^{-1} \, \overline{f_2} \, d\xi,
 \quad \textrm{for all $g \in G$}.
\]

\vspace{0.1cm}
\subsubsection*{\textsc{$m_K$-decoupled pairs}}
\vspace{0.1cm}

Let $\Delta: G \ra \bR^*_{+}$ be a continuous homomorphism. We say that a $\sigma$-finite and $G$-quasi-invariant  measure $\xi'$ on $X$ is \emph{$\Delta$-invariant}
if 
\[
g\xi' = \Delta(g)\xi', \quad \textrm{for all $g \in G$},
\] 
or equivalently, if $\frac{dg\xi'}{d\xi'} = \Delta(g)$, almost everywhere. It follows from the discussion above
that if $\xi$ is equivalent to a $\Delta$-invariant $\sigma$-finite measure $\xi'$, then $(L^p(X,\xi),\sigma_{p,\xi})$ is \emph{not} $L^p$-irreducible if and only
if there are non-zero $F_1 \in L^p(X,\xi')$ and $F_2 \in L^{p'}(X,\xi')$ such that
\[
\int_X F_1 \circ g^{-1} \, \overline{F_2} \, d\xi' = 0, \quad \textrm{for all $g \in G$}.
\]
Let us from now on assume that $X$ is a locally compact and second countable space, and that $G$ acts jointly continuously on $X$ and
there is a $\Delta$-invariant $\sigma$-finite Borel measure $\xi'$ on $X$, which is equivalent to $\xi$. We say that a pair $(F_1,F_2)$ of non-zero \emph{compactly supported} continuous functions on $X$ are \emph{$\xi'$-decoupled} if 
\[
\int_X F_1 \circ g^{-1} \, \overline{F_2} \, d\xi' = 0, \quad \textrm{for all $g \in G$}.
\]
Since $F_1$ and $F_2$ are compactly supported and continuous, they belong to $L^p(K,\xi')$ for every $p \in [1,\infty]$.
The following lemma is immediate from our discussion above.
\begin{lemma}
\label{lemma_decopuled}
If a $\xi'$-decoupled pair exists, then $(L^p(X,\xi),\sigma_{\xi,p})$ is \emph{not} $L^p$-irreducible for any $p \in [1,\infty)$.
\end{lemma}

Let us now specialize to the setting of Theorem \ref{thm3}, where
\[
G = H \qand X = K \qand \xi = \nu \qand \xi' = m_K,
\]
where $\nu$ is the unique $\tau$-stationary and $\theta_\tau$-stationary probability measure on $K$ (and equivalent to $m_K$ by \textsc{(II)}). \\

Set $\Delta_K(x,\varpi^n) = q^n$, where $q = |\cO/\cP|$ is a prime number. A straightforward calculation shows that $m_K$ is 
$\Delta_K$-invariant. To prove Theorem \ref{thm3}, it suffices by Lemma \ref{lemma_decopuled} to show that $(K,m_K)$ admits a $m_K$-decoupled pair. This 
is done in the  following lemma. Let $K^*$ denote the multiplicative group of $K$, and write
\[
\cO^* = \{ x \in \cO \, \mid \, |x| = 1 \big\}
\]
for the multiplicative group of the open sub-ring $\cO \subset K$, and $\widehat{K}$ for the (additive) dual of $K$, i.e. the multiplicative group of 
continuous homomorphisms from $(K,+)$ into the circle group $\bS^1 \subset \bC^*$.

\begin{lemma}
\label{z1z2}
For all $z_1, z_2 \in \cO^*$ and $\lambda \in \widehat{K}$ such that 
\[
|z_1 - z_2| \geq \frac{1}{q}  \qand \lambda |_{\varpi \cO} \neq 1,
\] 
we have
\[
\int_K F_{z_1}(y+\varpi^m x) \overline{F_{z_2}(x)} \, dm_K(x) = 0, \quad \textrm{for all $(y,\varpi^m) \in H$},
\]
where $F_{z}(x) = \lambda(xz) \, \chi_{\cO}(x)$ for $z \in K^*$. In particular, $(F_{z_1},F_{z_2})$ is a $m_K$-decoupled pair.
\end{lemma}

\begin{remark}
For example, the assumptions on $z_1$ and $z_2$ are satisfied for
\[
z_1 = 1 \qand z_2 = 1 + \varpi.
\]
Furthermore, if $\lambda'$ is a non-trivial character on $K$ (which always exists), then it must be non-trivial on some subgroup of $K$ of the form $\varpi^{-N} \cO$ for some $N \geq 0$ (since these subgroups exhaust $K$). If we set 
\[
\lambda(x) = \lambda'(x \varpi^{-(N+1)}), \quad \textrm{for $x \in K$}, 
\]
then $\lambda$ is a character on $K$ which is non-trivial on $\varpi \cO$.
\end{remark}

\begin{proof}
We first note that for all $z \in K$ and $\lambda \in \widehat{K}$,
\vspace{0.1cm}
\[
F_{z}(y+\varpi^m x) = \lambda(yz) \lambda(x \varpi^m z) \, \chi_{\varpi^{-m}(\cO-y)}(x), \quad \textrm{for all $y \in K$ and $m \in \bZ$},
\]
\vspace{0.1cm}
whence, for all $y, z_1, z_2 \in K$ and $m \in \bZ$,
\vspace{0.1cm}
\begin{eqnarray}
\int_K F_{z_1}(y+\varpi^m x) \overline{F_{z_2}(x)} \, dm_K(x)
&=&
\lambda(yz_1) \int_{K} \lambda(x(\varpi^m z_1 - z_2)) \chi_{\varpi^{-m}(\cO-y) \cap \cO}(x) \, dm_K(x) \nonumber \\[0.2cm]
&=&
\lambda(yz_1) \int_{\cO_{y,m}} \lambda(x w_{m}) \, dm_K(x), \nonumber
\end{eqnarray}
\vspace{0.1cm}
where 
\[
\cO_{y,m} = \varpi^{-m}(\cO-y) \cap \cO \qand w_m = \varpi^m z_1 - z_2.
\] 
In what follows, we shall assume that
\[
z_1, z_2 \in \cO^* \qand |z_1 - z_2| \geq \frac{1}{q} \qand \lambda|_{\varpi \cO} \neq 1,
\]
and show that $(F_{z_1},F_{z_2})$ is a $m_K$-decoupled pair. Note that the first two assumptions imply that $w_m \neq 0$ for all $m \in \bZ$. \\

To prove the lemma, it clearly suffices to show that 
\begin{equation}
\label{suffices}
\int_{\cO_{y,m}} \lambda_m(x) \, dm_K(x) = 0, \quad \textrm{for all $y \in K$ and $m \in \bZ$},
\end{equation}
where $\lambda_m$ is the (non-trivial) character on $K$ given by $\lambda_m(x) = \lambda(w_m x)$, for $x \in K$. \\

If $\cO_{y,m} = \emptyset$, then \eqref{suffices} is automatic. Note that
\[
\cO_{y,m} \neq \emptyset \iff y \in \cO - \varpi^{m} \cO = 
\left\{ 
\begin{array}{cl}
\cO & \textrm{if $m \geq 0$} \\[0.2cm]
\varpi^{m} \cO & \textrm{if $m < 0$}
\end{array}
\right..
\]
We shall deal with the cases $m \geq 0$ and $m < 0$ separately.

\subsubsection*{\textsc{Case I: $m \geq 0$ and $y \in \cO$}}

In this case, $\cO_{y,m} = \cO$. To prove \eqref{suffices}, it suffices to show that $\lambda_m|_{\cO} \neq 1$ (since the Haar-integral of any non-trivial character on a compact abelian group is zero). Note that
\[
\lambda_m|_{\cO} \neq 1 \iff \lambda|_{w_m \cO} \neq 1.
\]
Since $\lambda|_{\varpi \cO} \neq 1$, it thus suffices to show that $\varpi \cO \subseteq w_m \cO$, or equivalently, $|w_m| \geq q^{-1}$. The case $m = 0$ holds by assumption; indeed
$|w_o| = |z_1-z_2| \geq q^{-1}$. If $m \geq 1$, then since $z_1, z_2 \in \cO^*$, and thus 
$|z_1| = |z_2| = 1$, we have
\[
1 = |z_2| = |w_m - \varpi^m z_1| \leq \max(|w_m|,\underbrace{|\varpi^m z_1|}_{< 1}) = |w_m|,
\]
by the ultra-metric triangle inequality.

\subsubsection*{\textsc{Case II: $m < 0$ and $y \in \varpi^m \cO$}}

In this case, 
\begin{align*}
\cO_{y,m}
&=
\varpi^{-m}(\cO-y) \cap \cO = \varpi^{-m}((\cO-y) \cap \varpi^m \cO) \\[0.2cm]
&=
\varpi^{-m}(\cO \cap (\varpi^m \cO + y) - y) = \varpi^{-m}(\cO \cap \varpi^m \cO - y) \\[0.2cm]
&=
 \varpi^{-m}(\cO - y) = \varpi^{-m}\cO - \varpi^{-m}y.
\end{align*}
Hence,
\[
\int_{\cO_{y,m}} \lambda_m(x) \, dm_K(x) = \lambda_m(-\varpi^{-m}y) \, \int_{\varpi^{-m}\cO} \lambda_m(x) \, dm_K(x),
\]
and thus to establish \eqref{suffices}, we only need to show that $\lambda_m|_{\varpi^{-m}\cO} \neq 1$(since the Haar-integral of any non-trivial character on a compact abelian group is zero). Note that
\[
\lambda_m|_{\varpi^{-m}\cO} \neq 1 \iff \lambda|_{w_m \varpi^{-m}\cO} \neq 1.
\]
Since $\lambda|_{\varpi \cO} \neq 1$, it thus suffices to establish the inclusion $\varpi \cO \subseteq w_m \varpi^{-m}\cO$, or equivalently, the lower bound $|w_m| \geq q^{-(m+1)}$. However, since $m < 0$ and $|z_1| = |z_2| = 1$, we have
\[
1 < q^{-m} = |\varpi^{m}z_1| = |w_m+z_2| \leq \max(|w_m|,\underbrace{|z_2|}_{=1}) = |w_m|,
\]
by the ultra-metric triangle inequality.

\end{proof}

\begin{remark}
The non-$L^2$-irreducibility of $\sigma_{m_K,2}$ can be proved abstractly by appealing to the Mackey Little Group Method, developed in 
the seminal paper \cite{Mac}, according to which it suffices (in order to establish \textsc{iv}) to show that the multiplicative group $\langle \varpi \rangle$ does not act 
ergodically on $\widehat{K}$. Indeed, any $\langle \varpi \rangle$-invariant measurable subset, which is neither null nor co-null with respect
to $m_{\widehat{K}}$, provides a non-trivial $H$-equivariant projection on $L^2(K,m_K)$ (which in particular implies that $\sigma_{m_K,2}$
is not $L^2$-irreducible). However, it is not immediately clear that this projection extends continuously to other $L^p$-spaces, which is why we 
have preferred to outline the more constructive approach above.
\end{remark}

\section{Proof of Theorem \ref{thm4}}

Throughout this section, we shall fix
\vspace{0.1cm}
\begin{itemize}
\item a discrete countable \emph{prime} measured group $(\Xi,\sigma_o)$ with finite entropy. \vspace{0.2cm}
\item a positive summable sequence $\beta = (\beta_k)$.
\end{itemize}
\vspace{0.1cm}
Our aim is to construct from $\sigma_o$ a probability measure $\tau_\beta$ on the direct sum 
\[
\Gamma := \bigoplus_{\bN} \Xi = \bigcup_{n=1}^\infty \Big( \prod_{k=1}^n \Xi\Big),
\]
(with the obvious inclusions), whose support generates $\Gamma$ as a semi-group,  such that 
\[
\BndEnt(\Gamma,\tau_\beta) = \Big\{ \sum_{k \in S} \beta_k \, \mid \, S \subseteq \bN \Big\}.
\]
The following lemma, which will be proved below, provides all of the necessary ingredients in the construction of $\tau_\beta$.
\begin{lemma}
\label{lemma_pq}
There exist
\vspace{0.1cm}
\begin{enumerate}
\item[\textsc{(i)}] a probability measure $\sigma$ on $\Xi$ with finite entropy $h(\sigma)$, whose support generates $\Xi$ as a semi-group, such 
that $\sigma_o * \sigma = \sigma * \sigma_o$, \vspace{0.2cm}
\item[\textsc{(ii)}] a summable sequence $(p_k)$ such that $p_k \in (0,1)$ for all $k$, \vspace{0.2cm}
\item[\textsc{(iii)}] a positive sequence $(q_k)$ such that 
\[
1 = q_1 > q_2 > \ldots \qand \lim_k q_k = 0,
\]
\end{enumerate}
so that $\beta_k = p_k q_k h(\sigma) $ for all $k$.
\end{lemma}

\subsection{Proof of Theorem \ref{thm4} assuming Lemma \ref{lemma_pq}} 

We denote by $(B_o,\nu_o)$ the Poisson boundary of $(\Xi,\sigma_o)$. Since $\sigma * \sigma_o = \sigma_o * \sigma$, \textsc{(II)} in Proposition 
\ref{prop_proxcrit} tells us that $(B_o,\nu_o)$ is also a maximal $\sigma$-proximal Borel $(\Xi,\sigma)$-space, which is prime by assumption. \\

For $k \geq 1$, we define $\sigma_k \in \Prob(\Xi)$ by
\[
\sigma_k = (1-p_k)\delta_e + p_k \sigma.
\] 
Since $\sigma * \sigma_k = \sigma_k * \sigma$, \textsc{(II)} Proposition \ref{prop_proxcrit} again tells us that $(B_o,\nu_o)$ is a maximal 
$\sigma_k$-proximal Borel $(\Xi,\sigma_k)$-space, whence the Poisson boundary of $(\Xi,\sigma_k)$. Furthermore, by
Lemma \ref{lemma_propEntropy},
\[
h_k := h_{\sigma_k}(B_o,\nu_o) = p_k h_{\sigma}(B_o,\nu_o) = p_k h(\sigma), \quad \textrm{for all $k$}.
\]
\vspace{0.3cm}
We now define the probability measure $\widetilde{\tau}$ on the direct product $\prod_{\bN} \Xi$ by
\[
\widetilde{\tau} = \sigma_1 \otimes \sigma_2 \otimes \ldots,
\]
Since $(p_k)$ is summable and $\sigma_k(e) \geq 1-p_k$ for all $k$,
\[
\widetilde{\tau}(e,e,\ldots) = \prod_{k=1}^\infty \sigma_k(e) \geq \prod_{k=1}^\infty (1-p_k) > 0.
\]
By \cite[Section 3.V]{BS}, this implies that $\widetilde{\tau}$ gives full measure to $\Gamma < \prod_{\bN} \Xi$. Since each $\sigma_k$ generates 
$\Xi$ as a semigroup, we see
that $\widetilde{\tau}$ is a spread-out probability measure on $\Gamma$. \\

Since $(B_o,\nu_o)$ is a prime Borel $(\Xi,\sigma_k)$-space for every $k$, the following proposition is now a special case of \cite[Theorem 3.6]{BS}.
\begin{proposition}
\label{prop_adelic}
Every $\widetilde{\tau}$-boundary is of the form 
\[
(B_I,\nu_I) := \prod_{k \in I} (B_o,\nu_o), \quad \textrm{for some $I \subseteq \bN$}.
\]
In particular, the Poisson boundary of $(\Gamma,\widetilde{\tau})$ is $(B_{\bN},\nu_{\bN})$.
\end{proposition}
Let us now turn to the construction of $\tau_\beta$. We recall from Lemma \ref{lemma_pq} that
\begin{equation}
\label{bk}
\beta_k = p_k q_k \, h(\sigma), \quad \textrm{for all $k \geq 1$}.
\end{equation}
For $n \geq 1$, we set $\alpha_n = q_n - q_{n+1} > 0$, so that
\begin{equation}
\label{qk}
\sum_{n=1}^\infty \alpha_n = 1 \qand \sum_{n=k}^\infty \alpha_n = q_k \enskip \textrm{for all $k \geq 1$},
\end{equation}
and define $\tau_\beta \in \Prob(\Gamma)$ by
\[
\tau_\beta = \sum_{n=1}^\infty \, \alpha_n \, \sigma_1 \otimes \cdots \otimes \sigma_n,
\]
with the obvious interpretation of each term as a probability measure on $\Gamma$. Clearly, 
\[
\tau_\beta * \widetilde{\tau} = \widetilde{\tau} * \tau_\beta,
\]
so by \textsc{(II)} in Proposition \ref{prop_proxcrit}, every $\widetilde{\tau}$-boundary is a $\tau_\beta$-boundary (and vice versa).
Hence, by Proposition \ref{prop_adelic} 
\[
\BndEnt(\Gamma,\tau_\beta) = \Big\{ h_{\tau_\beta}(B_I,\nu_I) \, \mid \, I \subseteq \bN \Big\}.
\]
It remains for us to compute $h_{\tau_\beta}(B_I,\nu_I)$ for every $I \subseteq \bN$. By Lemma \ref{lemma_propEntropy} and after some simple 
rearrangements and applying \eqref{bk} and \eqref{qk}, we see that
\begin{eqnarray*}
h_{\tau_\beta}(B_I,\nu_I) 
&=& 
\sum_{n=1}^\infty \alpha_n \, h_{\sigma_1 \otimes \cdots \otimes \sigma_n}(B_{I \cap [1,n]},\nu_{I \cap [1,n]})= \sum_{n=1}^\infty \alpha_n \Big( \sum_{k \in I \cap [1,n]} h_k \Big) \\
&=&
\sum_{k \in I}  \Big( \sum_{n = k}^\infty \alpha_n \Big) \, h_k = \sum_{k \in I} p_k \, q_k h(\sigma) = \sum_{k \in I} \beta_k,
\end{eqnarray*}
which finishes the proof.

\subsection{Proof of Lemma \ref{lemma_pq}}

We begin by stating the following simple lemma, whose proof is left to the reader.

\begin{lemma}
For every positive summable sequence $(\beta_k)$, there is a strictly increasing positive sequence $(w_k)$ with $w_k \ra \infty$ as $k \ra \infty$
such that $\sum_{k=1}^\infty \beta_k \, w_k < \infty$.
\end{lemma}
 
Let us now fix a positive summable sequence $(\beta_k)$, as well as a strictly increasing positive sequence $(w_k)$ as in the lemma above. Upon scaling
the sequence $(w_k)$ we can clearly ensure that
\[
\sum_{k=1}^\infty \beta_k \, w_k < \infty \qand \beta_k w_k \in (0,1) \enskip \textrm{for all $k \geq 1$}.
\]
Given a probability measure $\sigma$ on $\Xi$ with finite positive entropy $h(\sigma)$, we now set 
\[
p_k = \beta_k \, w_k \qand q_k = \frac{1}{h(\sigma) \, w_k}.
\] 
Then $(q_k)$ is a strictly decreasing positive sequence with $q_k \ra 0$ as $k \ra \infty$, and thus to establish Lemma \ref{lemma_pq} 
it remains to prove that we can always find $\sigma \in \Prob(\Xi)$ such that
\[
\sigma * \sigma_o = \sigma_o * \sigma \qand h(\sigma) = \frac{1}{w_1}.
\]
To do this, let us define 
\[
\sigma_{\eps,N} = (1-\eps)\delta_e + \eps \sigma_o^{*N}, \quad \textrm{for $\eps \in (0,1)$ and $N \geq 1$}.
\]
By Lemma \ref{lemma_propEntropy}, $h(\sigma_{\eps,N}) = \eps N h(\sigma_o)$. If we choose $\eps$ and $N$ so 
that
\[
\eps N h(\sigma_o) = \frac{1}{w_1},
\]
we see that $h(\sigma_{\eps,N}) = \frac{1}{w_1}$. Clearly $\sigma := \sigma_{N,\eps}$ is spread-out (since $\sigma_o$ is spread-out), and
\[
\sigma * \sigma_o = \sigma_o * \sigma,
\]
which finishes the proof.

%For $N \geq 1$, let
%\[
%\Delta_N = \big\{ \underline{t} = (t_1,\ldots,t_N) \in (0,1)^N \, : \, t_1 + \ldots + t_N =1 \big\},
%\]
%and consider the function $r_{N,\eps} : \Delta_N \ra (0,\infty)$ defined by
%\[
%r_{N,\eps}(\underline{t}) = h(t_1 \sigma_o + t_2 \sigma_\eps^{*2} + \ldots + t_N \sigma_\eps^{*N}) = \Big(\sum_{l=1}^N l t_l\Big)h(\sigma_\eps) 
%= \eps \Big(\sum_{l=1}^N l t_l\Big)h(\sigma_o),
%\]
%where we in the last two identities have used Lemma \ref{lemma_propEntropy}. By connectedness of $\Delta_N$, the range of $r_{N,\eps}$ is the 
%whole open interval $(\eps h(\sigma_o),\eps N h(\sigma_o))$. In particular, we 
%can choose appropriate $\eps$ and $N$ so that for some $\underline{t}' = (t'_1,\ldots,t'_N) \in \Delta_N$, 
%we have $r_{N,\eps}(\underline{t}') = \frac{1}{w_1}$. Hence, if we set
%\[
%\sigma = t'_1 \sigma_\eps + t'_2 \sigma_\eps^{*2} + \ldots + t'_N \sigma_\eps^{*N} \in \Prob(\Xi),
%\]
%then, since $\sigma_\eps * \sigma_o = \sigma_o * \sigma_\eps$, we have
%\[
%\sigma * \sigma_o = \sigma_o * \sigma \qand h(\sigma) = \frac{1}{w_1},
%\]
%which finishes the proof.

\end{document}